\documentclass[12pt,english,reqno]{amsart}
\usepackage{charter}
\usepackage[T1]{fontenc}
\usepackage[latin9]{inputenc}
\usepackage{geometry}
\usepackage{color}
\usepackage{babel}
\usepackage{prettyref}
\usepackage{amstext}
\usepackage{amsthm}
\usepackage{amssymb}
\usepackage{setspace}
\usepackage{esint}
\usepackage[numbers]{natbib}
\usepackage[unicode=true,
 bookmarks=false,
 breaklinks=false,pdfborder={0 0 1},backref=page,colorlinks=true]
 {hyperref}
\hypersetup{pdftitle={Weak mixing for locally compact quantum groups},
 pdfauthor={Ami Viselter},
 pdfstartview={XYZ null null 1.25},citecolor=OliveGreen,urlcolor=blue}

\makeatletter
\numberwithin{equation}{section}
\theoremstyle{plain}
\newtheorem{thm}{\protect\theoremname}[section]
  \theoremstyle{definition}
  \newtheorem{example}[thm]{\protect\examplename}
  \theoremstyle{plain}
  \newtheorem{lem}[thm]{\protect\lemmaname}
  \theoremstyle{definition}
  \newtheorem{defn}[thm]{\protect\definitionname}
  \theoremstyle{plain}
  \newtheorem{prop}[thm]{\protect\propositionname}
  \theoremstyle{remark}
  \newtheorem{rem}[thm]{\protect\remarkname}
  \theoremstyle{plain}
  \newtheorem{cor}[thm]{\protect\corollaryname}
  \theoremstyle{plain}
  \newtheorem{question}[thm]{\protect\questionname}

\usepackage{eucal}
\let\mathcal=\CMcal
\usepackage{dsfont}
\usepackage{graphicx}
\usepackage{mathtools} 
\usepackage[all]{xy}
\usepackage[usenames,dvipsnames]{xcolor}

\let\ldash=\l

\newrefformat{eq}{\eqref{#1}}
\newrefformat{def}{Definition \ref{#1}}
\newrefformat{lem}{Lemma \ref{#1}}
\newrefformat{pro}{Proposition \ref{#1}}
\newrefformat{prop}{Proposition \ref{#1}}
\newrefformat{cor}{Corollary \ref{#1}}
\newrefformat{thm}{Theorem \ref{#1}}
\newrefformat{exa}{Example \ref{#1}}
\newrefformat{rem}{Remark \ref{#1}}
\newrefformat{sec}{Section \ref{#1}}
\newrefformat{sub}{Subsection \ref{#1}}
\newrefformat{conj}{Conjecture \ref{#1}}
\newrefformat{desc}{\ref{#1}}
\newrefformat{ques}{Question \ref{#1}}

\newcommand*{\doi}[1]{\href{http://dx.doi.org/#1}{doi: #1}}

\makeatletter
\let\orgdescriptionlabel\descriptionlabel
\renewcommand*{\descriptionlabel}[1]{%
  \let\orglabel\label
  \let\label\@gobble
  \phantomsection
  \edef\@currentlabel{#1}%
  \let\label\orglabel
  \orgdescriptionlabel{#1}%
}
\makeatother

\makeatother

  \providecommand{\corollaryname}{Corollary}
  \providecommand{\definitionname}{Definition}
  \providecommand{\examplename}{Example}
  \providecommand{\lemmaname}{Lemma}
  \providecommand{\propositionname}{Proposition}
  \providecommand{\questionname}{Question}
  \providecommand{\remarkname}{Remark}
\providecommand{\theoremname}{Theorem}

\begin{document}
\global\long\def\e{\varepsilon}
\global\long\def\N{\mathbb{N}}
\global\long\def\Z{\mathbb{Z}}
\global\long\def\Q{\mathbb{Q}}
\global\long\def\R{\mathbb{R}}
\global\long\def\C{\mathbb{C}}
\global\long\def\G{\mathbb{G}}
\global\long\def\HH{\mathbb{H}}

\global\long\def\H{\EuScript H}
\global\long\def\J{\mathcal{J}}
\global\long\def\K{\EuScript K}
\global\long\def\a{\alpha}
\global\long\def\be{\beta}
\global\long\def\l{\lambda}
\global\long\def\om{\omega}
\global\long\def\z{\zeta}
\global\long\def\Aa{\mathcal{A}}

\global\long\def\Ree{\operatorname{Re}}
\global\long\def\Img{\operatorname{Im}}
\global\long\def\linspan{\operatorname{span}}
\global\long\def\slim{\operatorname*{s-lim}}
\global\long\def\clinspan{\operatorname{\overline{span}}}
\global\long\def\co{\operatorname{co}}
\global\long\def\CAP{\mathrm{CAP}}
\global\long\def\presb#1#2{\prescript{}{#1}{#2}}

\global\long\def\tensor{\otimes}
\global\long\def\tensormin{\otimes_{\mathrm{min}}}
\global\long\def\tensorn{\overline{\otimes}}

\global\long\def\A{\forall}

\global\long\def\i{\mathrm{id}}

\global\long\def\one{\mathds{1}}
\global\long\def\tr{\mathrm{tr}}
\global\long\def\Ww{\mathds{W}}
\global\long\def\wW{\text{\reflectbox{\ensuremath{\Ww}}}\:\!}
\global\long\def\op{\mathrm{op}}
\global\long\def\WW{{\mathds{V}\!\!\text{\reflectbox{\ensuremath{\mathds{V}}}}}}

\global\long\def\Linfty#1{L^{\infty}(#1)}
\global\long\def\Lone#1{L^{1}(#1)}
\global\long\def\LoneStar#1{L_{*}^{1}(#1)}
\global\long\def\Ltwo#1{L^{2}(#1)}
\global\long\def\Cz#1{C_{0}(#1)}
\global\long\def\CzU#1{C_{0}^{\mathrm{u}}(#1)}
\global\long\def\CU#1{C^{\mathrm{u}}(#1)}

\global\long\def\linfty#1{\ell^{\infty}(#1)}
\global\long\def\lone#1{\ell^{1}(#1)}
\global\long\def\ltwo#1{\ell^{2}(#1)}
\global\long\def\Pol#1{\mathrm{Pol}(#1)}
\global\long\def\Ltwozero#1{L_{0}^{2}(#1)}
\global\long\def\Irred#1{\mathrm{Irred}(#1)}

\global\long\def\Ad#1{\mathrm{Ad}(#1)}
\global\long\def\VN#1{\mathrm{VN}(#1)}
\global\long\def\d{~\mathrm{d}}
\global\long\def\t{\mathrm{t}}
\global\long\def\tp{\xymatrix{*+<.7ex>[o][F-]{\scriptstyle \top}}
 }
\global\long\def\tpsmall{\xymatrix{*+<.5ex>[o][F-]{\scriptscriptstyle \top}}
 }

\title{Weak mixing for locally compact quantum groups}

\author{Ami Viselter}

\address{Department of Mathematics, University of Haifa, 31905 Haifa, Israel}

\email{aviselter@staff.haifa.ac.il}

\urladdr{http://math.haifa.ac.il/viselter}

\begin{abstract}
We generalize the notion of weakly mixing unitary representations
to locally compact quantum groups, introducing suitable extensions
of all standard characterizations of weak mixing to this setting.
These results are used to complement the noncommutative Jacobs--de
Leeuw--Glicksberg splitting theorem of Runde and the author {[}``Ergodic
theory for quantum semigroups'', J.~Lond.~Math.~Soc.~(2) 89 (2014)
941--959{]}. Furthermore, a relation between mixing and weak mixing
of state-preserving actions of discrete quantum groups and the properties
of certain inclusions of von Neumann algebras, which is known for
discrete groups, is demonstrated.
\end{abstract}

\maketitle

\section*{Introduction}

Weak mixing is an intrinsic part of ergodic theory. It was introduced
by Koopman and von Neumann in a specific setting \citep{Koopman_von_Neumann},
and then extended to amenable topological semigroups by Dye \citep{Dye__erg_mix_thm}
and to locally compact groups by Bergelson and Rosenblatt \citep{Bergelson_Rosenblatt__mixing}.
Lying between ergodicity and (strong) mixing, this notion comes in
various shapes and has an abundance of applications, an acclaimed
one being Furstenberg's proof of Szemer\'edi's theorem on arithmetic
progressions \citep{Furstenberg__erg_behav}. In operator algebras
it has noticeably played a central role in Popa's deformation/rigidity
theory, and in the study of singular masas in $\mathrm{II}_{1}$-factors
using the weak asymptotic homomorphism property and similar conditions;
see \citep{Popa__some_rig_Bernoulli,Popa__factors_Betti_num,Popa__strong_rig_malleable_I,Popa__strong_rig_malleable_II,Popa__cocycle_orbit_eq_superrig,Popa_Vaes__strong_rig_gen_Ber},
\citep{Sinclair_Smith__strong_singular,Robertson_Sinclair_Smith__strong_sing,Sinclair_Smith_White_Wiggins__strong_singular_MASAs,Jolissaint_Stalder__str_singular_MASAs,Cameron_Fang__Mukherjee__mix_subalg}
and their numerous sequels.

A related result is the famous Jacobs--de Leeuw--Glicksberg splitting
theorem \citep{Jacobs_ergodic_thy_1956,deLeeuw_Glicksberg__appl_almst_per_com}.
Considering a weakly almost periodic semigroup $\mathcal{S}$ of operators
on a Banach space, the notions of \emph{almost periodicity} and of
\emph{weak mixing} of vectors with respect to $\mathcal{S}$ are defined.
The theorem says, roughly, that under a suitable amenability condition,
$\mathcal{S}$ induces a splitting of the Banach space that is acted
on as the direct sum of the subspace of almost periodic\emph{ }vectors
and the subspace of weakly mixing vectors. An important example arises
from a dynamical system consisting of a topological semigroup acting
on a probability space by measure-preserving transformations. The
splitting theorem then applies to the Koopman representation of the
dynamical system.

This leads naturally to the question of finding noncommutative generalizations
of these results. Niculescu, Str\"{o}h and Zsid\'{o} \citep{Niculescu_Stroh_Zsido__noncmt_recur}
considered dynamical systems in which the object that was \emph{acted}
on was noncommutative, consisting of a von Neumann algebra $N$ and
an endomorphism of $N$ preserving a faithful normal state $\theta$;
this generalizes the classical setting of a measure-preserving map
on a probability space. A notion of almost periodic \emph{operators}
in $N$ with respect to the action was introduced, and it was proved
that the set of these operators is a von Neumann subalgebra $N^{\mathrm{AP}}$
of $N$. The Jacobs--de Leeuw--Glicksberg splitting has the form of
a $\theta$-preserving conditional expectation $E^{\mathrm{AP}}$
from $N$ onto $N^{\mathrm{AP}}$ \citep[Theorem 4.2]{Niculescu_Stroh_Zsido__noncmt_recur}.
The second direct summand, namely $\ker E^{\mathrm{AP}}$, has a weak
mixing property described in \citep[Proposition 5.5]{Niculescu_Stroh_Zsido__noncmt_recur}.
A generalization of these results to actions of more general (amenable)
topological semigroups $G$ was announced in \citep{Zsido__splitting_noncomm_dyn_sys}
(the above case is that of $G=\Z_{+}$). Subsequent works that address
weak mixing of actions of more general groups on von Neumann algebras
include Beyers, Duvenhage and Str{\"o}h \citep{Beyers_Duvenhage_Stroh__Szemeredi}
and Duvenhage \citep{Duvenhage__Bergelson_thm,Duvenhage__ergodicity_and_mixing,Duvenhage__relatively_ind_join};
see also Duvenhage and Mukhamedov \citep{Duvenhage_Mukhamedov__rel_erg_prop}.

Runde and the author undertook in \citep{Runde_Viselter_LCQGs_Ergodic_Thy}
to generalize the results of Niculescu, Str\"{o}h and Zsid\'{o}
to dynamical systems that were ``fully noncommutative'' in the sense
that the \emph{acting} object and the object that was \emph{acted}
on were both noncommutative. This means that a quantum semigroup
(a Hopf--von Neumann algebra) $\G$ acts on a von Neumann algebra
$N$ via an action $\a$ preserving a faithful normal state $\theta$.
We introduced the notions of completely almost periodic vectors and
operators with respect to $\a$. When $\G$ is amenable and co-amenable,
mild assumptions yield a Jacobs--de Leeuw--Glicksberg splitting theorem
as above: there exists a $\theta$-preserving conditional expectation
$E^{\CAP}$ from $N$ onto the set (indeed, von Neumann algebra) $N^{\CAP}$
of completely almost periodic operators \citep[Corollary 4.10]{Runde_Viselter_LCQGs_Ergodic_Thy}.
These mild assumptions are readily fulfilled when $\G$ is a locally
compact quantum group (LCQG) in the sense of Kustermans and Vaes \citep{Kustermans_Vaes__LCQG_C_star,Kustermans_Vaes__LCQG_von_Neumann}.
Nevertheless, while the completely almost periodic part was fully
taken care of, it was not clear at the time in what sense $\ker E^{\CAP}$
is ``weakly mixing''.

The purpose of the present paper is to develop a noncommutative theory
of weak mixing for unitary co-representations that, in particular,
will provide an answer to the question of the weak mixing nature of
$\ker E^{\CAP}$ from the previous paragraph. The proof for the case
$\G=\Z_{+}$ in \citep[Proposition 5.5]{Niculescu_Stroh_Zsido__noncmt_recur}
relies chiefly on the classical Koopman--von Neumann spectral mixing
theorem, characterizing the weakly mixing vectors with respect to
a contraction in a Hilbert space \citep[Theorem 2.3.4]{Krengel__book}.
It is tailored for a single contraction, i.e.~an action of $\Z_{+}$,
and does not extend to more general semigroups, let alone quantum
semigroups. Indeed, the general framework for weak mixing in ergodic
theory deals with groups. The setting of our main results will therefore
be that of LCQGs $\G$. It is neither more nor less restrictive than
that of \citep{Runde_Viselter_LCQGs_Ergodic_Thy} for we do not require
$\G$ to be co-amenable, and also amenability is required only for
some implications. \emph{Mixing} of unitary co-representations of
LCQGs was recently introduced by Daws, Fima, Skalski and White \citep{Daws_Fima_Skalski_White_Haagerup_LCQG}
and was successfully used to define the Haagerup property in this
setting.

Let us recall the main characterizations of weakly mixing group representations
and actions. We recommend the survey of Bergelson and Gorodnik \citep{Bergelson_Gorodnik__survey}
and the lecture notes of Austin \citep{Austin__erg_thy_lect_notes}
and Peterson \citep{Peterson__erg_thy_lect_notes}. Let $\pi$ be
a unitary representation of a locally compact group $G$ on a Hilbert
space $\H$. The following conditions are equivalent:
\begin{description}
\item [{(WM1)\label{desc:classical_weak_mixing_1}}] for every $\z_{1},\ldots,\z_{n}\in\H$
and $\e>0$ there is some $\gamma\in G$ such that $\left|\left\langle \pi(\gamma)\z_{i},\z_{j}\right\rangle \right|<\e$
for every $1\leq i,j\leq n$;
\item [{(WM2)}] the tensor product $\pi\tensor\overline{\pi}$, of $\pi$
and its conjugate representation $\overline{\pi}$, is ergodic;
\item [{(WM3)\label{desc:classical_weak_mixing_3}}] ``continuity of the
spectrum'' / triviality of the Kronecker factor: $\pi$ has no finite-dimensional
sub-representations.
\end{description}
In fact, if $\pi$ is weakly mixing, then its tensor product by any
other unitary representation of $G$ is also ergodic. Under additional
hypotheses, e.g.~amenability of $G$, formally stronger conditions
are equivalent to weak mixing.

An important, intuitive special case comes from (classical) dynamical
systems $(X,\mathbb{A},\mu,G,T)$ consisting of a probability space
$(X,\mathbb{A},\mu)$, a locally compact group $G$ and a measure-preserving
action $T=\left(T_{\gamma}\right)_{\gamma\in G}$ of $G$ on $X$.
The \emph{Koopman representation} of the system is the unitary representation
$\pi$ of $G$ on $\Ltwo{X,\mathbb{A},\mu}$ given by $\pi(\gamma)f:=f\circ T_{\gamma^{-1}}$
($\gamma\in G$, $f\in\Ltwo{X,\mathbb{A},\mu}$). Let $\Ltwozero{X,\mathbb{A},\mu}:=\left\{ f\in\Ltwo{X,\mathbb{A},\mu}:\int_{X}f\d\mu=0\right\} $.
We say that the system is weakly mixing if the restriction of the
Koopman representation to $\Ltwozero{X,\mathbb{A},\mu}$ satisfies
one, hence all, of conditions \prettyref{desc:classical_weak_mixing_1}--\prettyref{desc:classical_weak_mixing_3}
above. The first two take the following simpler forms:
\begin{enumerate}
\item for every $A_{1},\ldots,A_{n}\in\mathbb{A}$ and $\e>0$ there is
some $\gamma\in G$ such that\linebreak{}
 $\left|\mu(A_{i}\cap T_{\gamma}A_{j})-\mu(A_{i})\mu(A_{j})\right|<\e$
for every $1\leq i,j\leq n$;
\item the product system\emph{ }$(X\times X,\mathbb{A}\times\mathbb{A},\mu\times\mu,G,T\times T)$,
where $T\times T:=\left(T_{\gamma}\times T_{\gamma}\right)_{\gamma\in G}$,
is ergodic.
\end{enumerate}

The paper is structured as follows. In \prettyref{sec:equivalent_conditions}
we introduce the notion of weakly mixing unitary co-representations
of LCQGs. All classical characterizations of this property (see above)
are generalized in the paper's main result, \prettyref{thm:wm_conditions}.
It is subsequently used to complete the von Neumann algebraic Jacobs--de
Leeuw--Glicksberg splitting theorem of \citep{Runde_Viselter_LCQGs_Ergodic_Thy}
in \prettyref{sub:Jacobs_de-Leeuw_Glicksberg}. In \prettyref{sub:mixing_in_crossed_products}
we study the mixing and weak mixing properties of inclusions of von
Neumann algebras arising from a discrete quantum group acting on a
von Neumann algebra, generalizing part of results of Jolissaint and
Stalder \citep{Jolissaint_Stalder__str_singular_MASAs} and Cameron,
Fang and Mukherjee \citep{Cameron_Fang__Mukherjee__mix_subalg}. \prettyref{sec:open_questions}
ends the paper with several open questions.

\section{Preliminaries}

Let us begin with a few conventions. All Hilbert spaces in this paper
are complex. For $\z,\eta$ in a Hilbert space $\H$, we let $\om_{\z,\eta}$
stand for the functional in $B(\H)_{*}$ given by $x\mapsto\left\langle x\z,\eta\right\rangle $,
and let $\om_{\z}:=\om_{\z,\z}$. By $\z^{*}$ we mean the functional
$\left\langle \cdot,\z\right\rangle \in\H^{*}$. We write $\one$
for the unit of a $C^{*}$-algebra (if exists), and $\i$ for the
identity map over $C^{*}$-algebras. The left and the right absolute
values of an element $x$ of a $C^{*}$-algebra are $\left|x\right|:=(x^{*}x)^{1/2}$
and $\left|x\right|_{\mathrm{r}}:=(xx^{*})^{1/2}$, respectively.
Representations of $C^{*}$-algebras are assumed to be nondegenerate.
The flip map $a\tensor b\mapsto b\tensor a$ at the algebra level
is denoted by $\sigma$. The symbols $\tensor,\tensormin$ and $\tensorn$
are used for the Hilbert space, minimal (spatial) $C^{*}$-algebraic,
and normal spatial tensor products, respectively.

Let $N$ be a von Neumann algebra. For a weight $\theta$ on $N$,
we denote by $\left(\Ltwo{N,\theta},\Lambda_{\theta}\right)$ the
associated GNS construction. When $\theta$ is normal, semi-finite
and faithful (n.s.f.), we denote by $\nabla_{\theta}$ and $J_{\theta}$
the modular operator and modular conjugation of $\theta$, respectively,
both acting on $\Ltwo{N,\theta}$, and by $\sigma^{\theta}$ the modular
automorphism group of $\theta$ \citep{Stratila__mod_thy,Takesaki__book_vol_2}. 

\emph{Locally compact quantum groups} (LCQGs) are a far-reaching generalization
of locally compact groups. Their axiomatization, which is the product
of a long list of works that go back to the seventies, was introduced
by Kustermans and Vaes \citep{Kustermans_Vaes__LCQG_C_star,Kustermans_Vaes__LCQG_von_Neumann},
and an equivalent one by Masuda, Nakagami and Woronowicz \citep{Masuda_Nakagami_Woronowicz}.
A LCQG is a pair $\G=\left(\Linfty{\G},\Delta\right)$ satisfying
the following conditions:
\begin{enumerate}
\item $\Linfty{\G}$ is a von Neumann algebra;
\item $\Delta$ is a co-multiplication, namely a unital normal $*$-homomorphism
$\Delta:\Linfty{\G}\to\Linfty{\G}\tensorn\Linfty{\G}$ that is co-associative
in the sense that $(\Delta\tensor\i)\Delta=(\i\tensor\Delta)\Delta$;
\item There exist n.s.f.~weights $\varphi,\psi$ on $\Linfty{\G}$, called
the left and right Haar weights, respectively, that satisfy 
\[
\begin{gathered}\varphi((\om\tensor\i)\Delta(x))=\varphi(x)\om(\one)\quad\text{for all }\om\in\Linfty{\G}_{*}^{+},x\in\Linfty{\G}^{+}\text{ with }\varphi(x)<\infty,\\
\psi((\i\tensor\om)\Delta(x))=\psi(x)\om(\one)\quad\text{for all }\om\in\Linfty{\G}_{*}^{+},x\in\Linfty{\G}^{+}\text{ with }\psi(x)<\infty.
\end{gathered}
\]

\end{enumerate}
We set $\nabla:=\nabla_{\varphi}$ and $J:=J_{\varphi}$. The predual
$\Linfty{\G}_{*}$ of $\Linfty{\G}$ is denoted by $\Lone{\G}$. It
is a Banach algebra when equipped with the convolution product $\om_{1}*\om_{2}:=(\om_{1}\tensor\om_{2})\Delta$,
$\om_{1},\om_{2}\in\Lone{\G}$.

The \emph{dual} of $\G$ is a LCQG denoted by $\hat{\G}=(\Linfty{\hat{\G}},\hat{\Delta})$.
We will not elaborate on the precise construction of $\hat{\G}$,
but give only a few details we shall need. The objects associated
with the dual will be denoted by adding a hat to the relevant notation,
e.g. $\hat{\varphi},\hat{\nabla},\hat{J}$. The Hilbert spaces $\Ltwo{\Linfty{\G},\varphi}$
and $\Ltwo{\Linfty{\hat{\G}},\hat{\varphi}}$ are canonically isometrically
isomorphic, allowing us to view both $\Linfty{\G}$ and $\Linfty{\hat{\G}}$
as acting standardly on the same Hilbert space $\Ltwo{\G}$ \citep{Haagerup__standard_form}.

The \emph{left regular co-representation} of $\G$ is a unitary $W\in\Linfty{\G}\tensorn\Linfty{\hat{\G}}$
satisfying $\Delta(x)=W^{*}(\one\tensor x)W$ for every $x\in\Linfty{\G}$
and $(\Delta\tensor\i)(W)=W_{13}W_{23}$, where the subscript numbers
are the customary leg numbering. Its dual object is just $\hat{W}=\sigma(W^{*})$.
The space $\Cz{\G}:=\clinspan^{\left\Vert \cdot\right\Vert }\{(\i\tensor\om)(W):\om\in\Lone{\hat{\G}}\}$
is a weakly dense $C^{*}$-subalgebra of $\Linfty{\G}$. It satisfies
$\Delta(\Cz{\G})\subseteq M(\Cz{\G}\tensormin\Cz{\G})$, where $M$
stands for the multiplier algebra. We have $W\in M(\Cz{\G}\tensormin\Cz{\hat{\G}})$.

The \emph{antipode} of $\G$ is a $*$-ultrastrongly closed, densely
defined operator $S$ over $\Linfty{\G}$ with the property \prettyref{eq:antipode_rep}
below. It admits a polar decomposition $S=R\circ\tau_{-i/2}$, where
the \emph{unitary antipode} $R$ is an anti-automorphism of $\Linfty{\G}$
and the \emph{scaling group} $\tau=\left(\tau_{t}\right)_{t\in\R}$
is a group of automorphisms of $\Linfty{\G}$, and by $\tau_{-i/2}$
we mean the analytic generator of $\tau$ at the point $-i/2$ \citep{Cioranescu_Zsido__analytic_gen,Zsido__spectral_erg_prop}.
We have the useful formulas 
\begin{equation}
\sigma(R\tensor R)\Delta=\Delta R\label{eq:R_Delta}
\end{equation}
and $R(x)=\hat{J}x^{*}\hat{J}$, $\tau_{t}(x)=\hat{\nabla}^{it}x\hat{\nabla}^{-it}$
for every $x\in\Linfty{\G}$, $t\in\R$.

We mention several types of LCQGs. First, \emph{compact} quantum groups
were introduced by Woronowicz \citep{Woronowicz__symetries_quantiques}
(see also Maes and Van Daele \citep{Maes_van_Daele__notes_CQGs}).
We say that $\G$ is compact if $\Cz{\G}$ is unital. This is equivalent
to either of the Haar weights being finite. In this case, the Haar
weights are, in fact, equal after being normalized to states, and
the common value is called the \emph{Haar state} of $\G$. We write
$C(\G)$ for $\Cz{\G}$ and $\Irred{\G}$ for the set of equivalence
classes of irreducible unitary co-representations of $\G$ (see below).
Second, \emph{discrete} quantum groups were introduced by Effros and
Ruan \citep{Effros_Ruan__discrete_QGs} and by Van Daele \citep{Van_Daele__discrete_CQs}.
We say that $\G$ is discrete if $\hat{\G}$ is compact. In this case,
we write $c_{0}(\G),\linfty{\G},\lone{\G}$ for $\Cz{\G},\Linfty{\G},\Lone{\G}$,
respectively. We have $\linfty{\G}=\ell^{\infty}-\bigoplus_{\gamma\in\Irred{\hat{\G}}}M_{N(\gamma)}$
and $c_{0}(\G)=c_{0}-\bigoplus_{\gamma\in\Irred{\hat{\G}}}M_{N(\gamma)}$,
where $N(\gamma)$ is the dimension of $\gamma$. There is a distinguished
central minimal projection $p\in c_{0}(\G)$ such that $\Delta(a)(\one\tensor p)=a\tensor p$
for all $a\in\linfty{\G}$. See Runde \citep{Runde__charac_compact_discr_QG}
for more information and proofs.

\emph{Kac algebras} \citep{Enock_Schwartz__book} were introduced
by Enock and Schwartz and by Kac and Vainerman. These are precisely
the LCQGs with a trivial scaling group ($\tau=\i$) that satisfy $\sigma^{\varphi}=\sigma^{\psi}$.
\begin{example}
Every locally compact group $G$ induces two LCQGs. The first, which
is identified with $G$, has $\Cz G,\Linfty G$ as its underlying
$C^{*}$- and von Neumann algebras, and $(\Delta(f))(t,s):=f(ts)$
for $f\in\Linfty G$ and $t,s\in G$. The Haar weights are given by
integration against the Haar measures. Its dual LCQG $\hat{G}$ has
$\Linfty{\hat{G}}=\VN G$, the left von Neumann algebra of $G$ generated
by the left translation operators $\left\{ \lambda(g):g\in G\right\} $
over $\Ltwo G$, $\Cz{\hat{G}}=C_{\mathrm{r}}^{*}(G)$, and $\hat{\Delta}$
is determined by satisfying $\hat{\Delta}(\lambda(g))=\lambda(g)\tensor\lambda(g)$
for each $g\in G$. The left and right Haar weights of $\hat{G}$
are both equal to the Plancherel weight on $\VN G$ \citep[Section VII.3]{Takesaki__book_vol_2}.
If $G$ is abelian, the quantum duality between $G$ and $\hat{G}$
reduces, up to unitary equivalence, to the Pontryagin duality.
\end{example}
A \emph{co-representation} of $\G$ on a Hilbert space $\H$ is an
operator $U\in B(\H)\tensorn\Linfty{\G}$ such that $(\i\tensor\Delta)(U)=U_{12}U_{13}$.
A closed subspace $\H_{0}\subseteq\H$ is called \emph{invariant}
under $U$ if the projection $p_{\H_{0}}$ of $\H$ onto $\H_{0}$
satisfies $U(p_{\H_{0}}\tensor\one)=(p_{\H_{0}}\tensor\one)U(p_{\H_{0}}\tensor\one)$.
The operator $U_{0}:=U(p_{\H_{0}}\tensor\one)\in B(\H_{0})\tensorn\Linfty{\G}$
is then a co-representation of $\G$ on $\H_{0}$. We say that $U_{0}$
is a \emph{sub-representation} of $U$, or that $U$ \emph{contains}
$U_{0}$, and write $U_{0}\leq U$. A vector $\z\in\H$ is said to
be \emph{invariant} under $U$ if the projection $p_{\z}$ of $\H$
onto $\C\z$ satisfies $U(p_{\z}\tensor\one)=p_{\z}\tensor\one$.
We say that $U$ is \emph{ergodic} if it has no nonzero invariant
vector. If $U,V$ are co-representations of $\G$ on Hilbert spaces
$\H,\K$, respectively, their \emph{tensor product} is the co-representation
$U\tp V:=U_{13}V_{23}$ of $\G$ on $\H\tensor\K$. 

\emph{Unitary} co-representations $U$ of $\G$ on $\H$ have additional
useful features. They satisfy $U\in M(\mathbb{K}(\H)\tensormin\Cz{\G})$,
where $\mathbb{K}(\H)$ is the $C^{*}$-algebra of all compact operators
over $\H$. Moreover, 
\begin{equation}
(\i\tensor S)(U)=U^{*},\label{eq:antipode_rep}
\end{equation}
that is, for every $\om\in B(\H)_{*}$ we have $(\om\tensor\i)(U)\in D(S)$
and $S((\om\tensor\i)(U))=(\om\tensor\i)(U^{*})$. As $S=R\circ\tau_{-i/2}$,
this means that formally $(\i\tensor\tau_{-i/2})(U)=(\i\tensor R)(U^{*})$,
which should be understood similarly. If a closed subspace $\H_{0}$
is invariant under $U$, then in the above notation, $U$ commutes
with $p_{\H_{0}}\tensor\one$, for instance by the next paragraph.
The associated sub-representation of $\G$ on $\H_{0}$ is therefore
unitary.

There is also the \emph{universal} face of $\G$ \citep{Kustermans__LCQG_universal}.
It consists of a $C^{*}$-algebra $\CzU{\G}$, a canonical surjective
$*$-homomorphism $\pi:\CzU{\G}\to\Cz{\G}$ and a co-multiplication
$\Delta^{\mathrm{u}}:\CzU{\G}\to M(\CzU{\G}\tensormin\CzU{\G})$ such
that $(\pi\tensor\pi)\Delta^{\mathrm{u}}=\Delta\pi$. Doing this construction
also for $\hat{\G}$, the left regular co-representation $W$ lifts
to a universal left regular co-representation, which is a unitary
$\wW\in M(\Cz{\G}\tensormin\CzU{\hat{\G}})$ with $(\i\tensor\hat{\pi})(\wW)=W$.
The universality is expressed by the following fact: for every unitary
co-representation $U$ of $\G$ on a Hilbert space $\H$, there exists
a representation $\rho$ of $\CzU{\hat{\G}}$ on $\H$ such that $U=\sigma(\i\tensor\rho)(\wW)$.
(We require the burdensome flip since we are working with the left
regular co-representation, but with right unitary co-representations.)
In particular, the character $\epsilon$ of $\CzU{\G}$ corresponding
to the trivial representation $\one$ of $\hat{\G}$ is called the
\emph{co-unit} of $\G$ and satisfies $(\epsilon\tensor\i)\Delta^{\mathrm{u}}=\i$.
There are also universal versions $R^{\mathrm{u}},\tau^{\mathrm{u}}$
of $R,\tau$.

A \emph{left-invariant mean} for $\G$ is a state $m\in\Linfty{\G}^{*}$
satisfying $m(\om\tensor\i)\Delta=\om(\one)m$ for every $\om\in\Lone{\G}$.
Right-invariant and two-sided invariant means are defined similarly.
We say that $\G$ is \emph{amenable} if it possesses a left-invariant
mean. This is equivalent to $\G$ admitting a right-invariant, or
a two-sided invariant, mean. We say that $\G$ is \emph{co-amenable}
if the canonical surjective $*$-homomorphism $\pi:\CzU{\G}\to\Cz{\G}$
is injective. In this case, we identify $\CzU{\G}$ with $\Cz{\G}$.
This is equivalent to $\Lone{\G}$ having a bounded left or right
or two-sided approximate identity, so discrete quantum groups are
trivially co-amenable. For other equivalent conditions for amenability
and co-amenability, see B\'edos and Tuset \citep{Bedos_Tuset_2003}.

A unitary co-representation $U$ of $\G$ on $\H$ is called \emph{mixing}
\citep[Definition 4.1]{Daws_Fima_Skalski_White_Haagerup_LCQG} if
for every $\om\in B(\H)_{*}$, $(\om\tensor\i)(U)\in\Cz{\G}$. The
left regular co-representation of $\G$ is mixing by definition. Additionally,
all unitary co-representations of a compact quantum group are mixing.

A (right) \emph{action} of $\G$ on a von Neumann algebra $N$ is
a unital normal $*$-homomorphism $\a:N\to N\tensorn\Linfty{\G}$
such that 
\[
(\a\tensor\i)\a=(\i\tensor\Delta)\a.
\]
For $\theta\in\Lone{\G}$, we say that $\a$ \emph{preserves} $\theta$
if $(\theta\tensor\i)\a=\theta(\cdot)\one$. If $\G$ co-amenable,
we adopt the convention of requiring that some bounded left approximate
identity $\left(\epsilon_{\lambda}\right)$ of the Banach algebra
$\Lone{\G}$ would satisfy 
\begin{equation}
(\i\tensor\epsilon_{\lambda})\a(a)\to a\text{ weakly}\qquad(\forall a\in N).\label{eq:action_co_amen_LCQG}
\end{equation}
If $\G$ is discrete, this is the same as asking that, with $p\in c_{0}(\G)$
being the distinguished central minimal projection, 
\begin{equation}
\a(a)(\one\tensor p)=a\tensor p\qquad(\forall a\in N).\label{eq:action_discrete_QG}
\end{equation}

\section{\label{sec:equivalent_conditions}Conditions for weak mixing}

In this section we introduce weak mixing of unitary co-representations
of LCQGs, generalizing conditions \prettyref{desc:classical_weak_mixing_1}--\prettyref{desc:classical_weak_mixing_3}
from the Introduction. There are two versions: weak mixing (\prettyref{def:weak_mixing})
and strict weak mixing (\prettyref{def:strict_weak_mixing}), which
coincide, e.g., for Kac algebras and for discrete quantum groups.
The first is the natural one for LCQGs. We establish the equivalence
of all conditions except for one implication, for which amenability
is assumed. The reason for that is the possible lack of certain invariant
means, as explored in depth by Das and Daws \citep{Das_Daws__quantum_Eberlein}.
The second version is not as natural, and indeed, less is known about
it without assuming amenability. The ``price'' that one has to pay
when working with the former version is that most of its characterizations
have two parts, and some involve unbounded co-representations.

\subsection{\label{sub:prelim_CAP}Preliminaries on complete almost periodicity}

We need several facts about notions discussed in \citep{Runde_Viselter_LCQGs_Ergodic_Thy}.
See also So{\ldash}tan \citep{Soltan__quantum_Bohr_comp} and Daws
\citep{Daws__remarks_quantum_Bohr_comp}.
\begin{lem}[{cf.~\citep[Proposition 4.11]{Soltan__quantum_Bohr_comp}}]
\label{lem:compactification_antipode}Let $\G$ be a LCQG. Suppose
that $C(\HH)$ is a unital $C^{*}$-subalgebra of $\Linfty{\G}$ such
that $\Delta(C(\HH))\subseteq C(\HH)\tensormin C(\HH)$ and $\HH:=(C(\HH),\Delta|_{C(\HH)})$
is a $C^{*}$-algebraic compact quantum group in the sense of \citep{Woronowicz__symetries_quantiques,Maes_van_Daele__notes_CQGs}.
Then the antipode, unitary antipode and scaling group $S^{\HH},R^{\HH},\tau^{\HH}$
of $\HH$ are the restrictions of $S^{\G},R^{\G},\tau^{\G}$ to the
canonical dense Hopf $*$-algebra $\mathcal{A}$ of $\HH$. In particular,
$S^{\HH}$ is (norm) closable, $\tau^{\HH}$ extends to an automorphism
group of $C(\HH)$ and $R^{\HH}$ extends to an anti-automorphism
of $C(\HH)$.
\end{lem}
Such $\HH$ is called a \emph{compactification} of $\G$. The point
is that as $\HH$ is not necessarily reduced in the sense of \citep{Kustermans_Vaes__LCQG_C_star}
(cf.~\citep[Subsection 6.2]{Daws__remarks_quantum_Bohr_comp}), the
above properties of the antipode are not automatic; cf.~\citep[Section 8]{Woronowicz__symetries_quantiques}
and \citep[Proposition 5.45]{Kustermans_Vaes__LCQG_C_star}.
\begin{proof}[Proof of \prettyref{lem:compactification_antipode}]
Every co-representation of $\HH$ is a co-representation of $\G$.
By \prettyref{eq:antipode_rep}, the antipode $S^{\HH}$ of $\HH$
agrees with $S^{\G}$ on the components of all irreducible unitary
co-representations of $\HH$, which linearly span $\mathcal{A}$.
Thus $S^{\HH}\subseteq S^{\G}$, and in particular $S^{\HH}$ is closable.
Let $u\in M_{n}\tensor C(\HH)$ be an irreducible unitary co-representation
of $\HH$. By the classical theory \citep{Woronowicz__symetries_quantiques},
$u\in D(\i\tensor(S^{\HH})^{2})$, and there is a strictly positive
matrix $F\in M_{n}$ such that $(\i\tensor\tau_{t}^{\HH})(u)=(F^{it}\tensor\one)u(F^{-it}\tensor\one)$
for every $t\in\R$ and $(\i\tensor(S^{\HH})^{2})(u)=(F\tensor\one)u(F^{-1}\tensor\one)$.
By the foregoing, $u\in D(\i\tensor(S^{\G})^{2})$ and 
\[
(\i\tensor\tau_{-i}^{\G})(u)=(\i\tensor(S^{\G})^{2})(u)=(F\tensor\one)u(F^{-1}\tensor\one).
\]
A classical complex analysis argument implies that $(\i\tensor\tau_{t}^{\G})(u)=(F^{it}\tensor\one)u(F^{-it}\tensor\one)$
for every $t$. Thus $\i\tensor\tau^{\G}$ and $\i\tensor\tau^{\HH}$
agree on $u$. Consequently, $\tau^{\HH}$ is the restriction of $\tau^{\G}$
to $\mathcal{A}$, and so $\tau^{\HH}$ extends to an automorphism
group of $C(\HH)$. The conclusion about $R$ follows.
\end{proof}
For the following definition see \citep[Definition 3.7]{Runde_Viselter_LCQGs_Ergodic_Thy},
in which the setting is a little different.
\begin{defn}
Let $U$ be a unitary co-representation of a LCQG $\G$ on a Hilbert
space $\H$. A vector $\z\in\H$ is \emph{completely periodic} with
respect to $U$ if it belongs to a finite-dimensional subspace $\H_{0}$
of $\H$, such that $U$ restricts to a sub-representation $u$ on
$\H_{0}$ whose transpose $u^{\t}$ is invertible. The closed linear
span $\H^{\CAP}$ of all completely periodic vectors is called the
subspace of \emph{completely almost periodic} vectors with respect
to $U$.
\end{defn}
Finite-dimensional unitary co-representations admitting an invertible
transpose are \emph{admissible} in the terminology of \citep{Soltan__quantum_Bohr_comp},
in which this notion is used in a broader sense. Proving the invertibility
of the transpose can be done as follows (cf.~\citep[Proposition 3.11]{Daws__remarks_quantum_Bohr_comp}).
Let $u$ be a finite-dimensional unitary co-representation of $\G$
on a Hilbert space $\H_{0}$. Let $n:=\dim\H_{0}$, view $B(\H_{0})$
as $M_{n}$ and denote by $A$ the natural anti-linear isomorphism
of $M_{n}$ given by $\left(a_{ij}\right)\mapsto\left(\overline{a_{ij}}\right)$.
Since the unitary antipode satisfies $R(x)=\hat{J}x^{*}\hat{J}$ for
every $x\in\Linfty{\G}$, from \prettyref{eq:antipode_rep} we get
$(\i\tensor\tau_{-i/2})(u)=(\i\tensor R)(u^{*})=(A\tensor\Ad{\hat{J}})(u^{\t})$.
Thus $u^{\t}$ is invertible if and only if $(\i\tensor\tau_{-i/2})(u)$
is. As $\tau_{-i/2}$ is multiplicative on its domain, this happens
when $u^{*}\in D(\i\tensor\tau_{-i/2})$, that is, $u\in D(\i\tensor\tau_{i/2})=D(\i\tensor S^{-1})$.
Conversely, if $u^{\t}$ is invertible, then the unital $C^{*}$-algebra
$C(\HH)$ generated by the components of $u$, together with the restriction
of $\Delta$ to $C(\HH)$, is a (not necessarily reduced) $C^{*}$-algebraic
compact quantum group $\HH$ by \citep{Woronowicz__remark_on_CQGs},
\citep[Proposition 3.8]{Maes_van_Daele__notes_CQGs}. From \prettyref{lem:compactification_antipode},
the scaling group $\tau^{\HH}$ of $\HH$ is a restriction of $\tau=\tau^{\G}$.
From the general theory \citep{Woronowicz__symetries_quantiques},
the components of $u$ are analytic for $\tau^{\HH}$, thus for $\tau^{\G}$.

Using the idea of the canonical Kac quotient of a compact quantum
group, it is proved in \citep[Subsection 4.3]{Soltan__quantum_Bohr_comp}
that, for a \emph{discrete} quantum group $\G$, the components of
all finite-dimensional unitary co-representations of $\G$ linearly
span the same subspace as those of the ones having an invertible transpose.
The previous paragraph thus implies the following; cf.~\citep[Corollary 6.6]{Daws__remarks_quantum_Bohr_comp}.
\begin{prop}
\label{prop:DQG_invertible_transpose}Every finite-dimensional unitary
co-representation of a discrete quantum group has an invertible transpose.
\end{prop}
It is unknown whether this holds true for all LCQGs; see \citep[Conjectures 7.1 and 7.2]{Daws__remarks_quantum_Bohr_comp}.

\subsection{$\overline{U}$ and $U'$}
\begin{defn}
\label{def:U_bar_U'}Let $\G$ be a LCQG and let $U\in B(\H)\tensorn\Linfty{\G}$
be a unitary co-representation of $\G$ on a Hilbert space $\H$.
\begin{enumerate}
\item (cf.~\citep{Bedos_Tuset_2003}) Fixing some anti-unitary $\J$ from
$\H$ onto another Hilbert space $\J\H$, consider the $*$-anti-isomorphism
$j:B(\H)\to B(\J\H)$ given by $j(x):=\J x^{*}\J^{*}$ for all $x\in B(\H)$,
and define the unitary co-representation of $\G$ \emph{conjugate}
to $U$ to be $\overline{U}:=(j\tensor R)(U)\in B(\J\H)\tensorn\Linfty{\G}$.
\item Let $U'$ be the formal object $(\i\tensor\tau_{-i/2})(U)=(\i\tensor R)(U^{*})$
(see \prettyref{eq:antipode_rep}). Rigorously, recalling that $B(\Lone{\G},B(\H))\cong B(B(\H)_{*},\Linfty{\G})$
canonically, we view $U'$ either as the element $\Lone{\G}\ni\om\mapsto(\i\tensor(\om\circ R))(U^{*})$
of $B(\Lone{\G},B(\H))$ or, equivalently, as the matching element
$B(\H)_{*}\ni\rho\mapsto\tau_{-i/2}((\rho\tensor\i)(U))=R((\rho\tensor\i)(U^{*}))$
of $B(B(\H)_{*},\Linfty{\G})$, thus defining the expressions $(\i\tensor\om)(U')$
and $(\rho\tensor\i)(U')$.
\end{enumerate}
\end{defn}
When the maps defining $U'$ happen to be completely bounded, we get
an element of $B(\H)\tensorn\Linfty{\G}$, as this operator space
is naturally identified with $\mathcal{CB}(\Lone{\G},B(\H))\cong\mathcal{CB}(B(\H)_{*},\Linfty{\G})$
\citep[Chapter 7]{Effros_Ruan__book}. For instance, when the scaling
group of $\G$ is trivial (e.g., if $\G$ is a Kac algebra), we have
$U'=U$.

Plainly, $\overline{U}$ depends on $\J$ only up to unitary equivalence.
Using \prettyref{eq:R_Delta}, one verifies that $\overline{U}$ is
indeed a unitary co-representation of $\G$ on $\J\H$, and that $U'$
is formally an unbounded co-representation of $\G$ on $\H$ in the
sense that $[\i\tensor((\om_{1}*\om_{2})\circ R)](U^{*})=(\i\tensor(\om_{1}\circ R))(U^{*})(\i\tensor(\om_{2}\circ R))(U^{*})$
for all $\om_{1},\om_{2}\in\Lone{\G}$.
\begin{lem}
\label{lem:U'_left_leg}Let $U$ be a unitary co-representation of
a LCQG $\G$ on a Hilbert space $\H$. For every $\z,\eta\in\H$,
we have
\begin{equation}
(\om_{\J\z,\J\eta}\tensor\i)(\overline{U})=R\left[(\om_{\eta,\z}\tensor\i)(U)\right]=(\om_{\z,\eta}\tensor\i)(U')^{*}.\label{eq:U'_left_leg__1}
\end{equation}
In addition, for every Hilbert space $\K$ and $\Xi\in\J\H\tensor\K$,
the operator $T_{\Xi}\in B(\H,\K)$ given by $T_{\Xi}\z:=((\J\z)^{*}\tensor\one)(\Xi)$,
$\z\in\H$, satisfies 
\begin{equation}
(\om_{(\one\tensor\xi^{*})(\Xi),\J\z}\tensor\i)(\overline{U}^{*})=(\om_{\z,T_{\Xi}^{*}\xi}\tensor\i)(U')\label{eq:U'_left_leg__2}
\end{equation}
for every $\z\in\H$ and $\xi\in\K$.\end{lem}
\begin{proof}
Equation \prettyref{eq:U'_left_leg__1} is obtained from the equality
$\om_{\J\z,\J\eta}\circ j=\om_{\eta,\z}$. Notice that $T_{\Xi}^{*}\xi=\J(\one\tensor\xi^{*})(\Xi)$
for each $\xi\in\K$. Hence, taking adjoints in \prettyref{eq:U'_left_leg__1}
with $\eta:=T_{\Xi}^{*}\xi$, we establish \prettyref{eq:U'_left_leg__2}.
\end{proof}
We now explain the nomenclature used in Definitions \ref{def:weak_mixing}
and \ref{def:strict_weak_mixing}.
\begin{defn}
\label{def:BR}Consider unitary co-representations $U,V$ of a LCQG
$\G$ on Hilbert spaces $\H,\K$, respectively. Denote by $\mathcal{BR}(V\tp U')$
the set of all closed subspaces $\mathcal{F}$ of $\K\tensor\H$ such
that, with $p_{\mathcal{F}}$ being the projection of $\K\tensor\H$
onto $\mathcal{F}$, we have $U_{23}(p_{\mathcal{F}}\tensor\one)\in D(\i\tensor\i\tensor\tau_{-i/2})$,
and $Z:=V_{13}(\i\tensor\i\tensor\tau_{-i/2})(U_{23}(p_{\mathcal{F}}\tensor\one))$
is a co-representation of $\G$ commuting with $p_{\mathcal{F}}\tensor\one$.
In particular, if $\mathcal{F}=\C\Xi$ for a vector $\Xi\in\K\tensor\H$
and $Z$ is equal to $p_{\C\Xi}\tensor\one$, we say that \emph{$\Xi$
is invariant under $V\tp U'$.}
\end{defn}
Evidently, if the scaling group of $\G$ is trivial (so that $U'=U$),
then $\K\tensor\H$ itself belongs to $\mathcal{BR}(V\tp U')$, and
invariance of a vector under $V\tp U'$ has the usual meaning\emph{.}

For general LCQGs, invariance of vectors under $V\tp U'$ has a more
concrete interpretation\emph{.}
\begin{lem}
\label{lem:V_U'_invariant_vector}Let $U,V$ be unitary co-representations
of a LCQG $\G$ on Hilbert spaces $\H,\K$, respectively, and $\Xi\in\K\tensor\H$.
Then $\Xi$ is invariant under $V\tp U'$ if and only if for all $\z\in\K$,
$\xi\in\H$ 
\begin{equation}
(\om_{(\z^{*}\tensor\one)(\Xi),\xi}\tensor\i)(U')=(\om_{(\one\tensor\xi^{*})(\Xi),\z}\tensor\i)(V^{*})\qquad(\forall\z\in\K,\xi\in\H).\label{eq:V_U'_invariant_vector}
\end{equation}
\end{lem}
\begin{proof}
Write $p_{\C\Xi}$ for the projection of $\K\tensor\H$ onto $\C\Xi$.
The implication $(\implies)$ is obtained by applying $\om_{\Xi,\z\tensor\xi}\tensor\i$
to both sides of the equality $(\i\tensor\i\tensor\tau_{-i/2})(U_{23}(p_{\C\Xi}\tensor\one))=V_{13}^{*}(p_{\C\Xi}\tensor\one)$.
For $(\impliedby)$, approximate $U_{23}(p_{\C\Xi}\tensor\one)$ by
operators of the form 
\[
\sum_{\a,\be}((f_{\be}\tensor e_{\a})\tensor\Xi^{*})\tensor(\om_{\Xi,f_{\be}\tensor e_{\a}}\tensor\i)(U_{23})=\sum_{\a,\be}((f_{\be}\tensor e_{\a})\tensor\Xi^{*})\tensor(\om_{(f_{\be}^{*}\tensor\one)(\Xi),e_{\a}}\tensor\i)(U),
\]
where $\left(e_{\a}\right),\left(f_{\be}\right)$ are orthonormal
bases of $\H,\K$, respectively, and the sums are finite. By assumption,
applying $\i\tensor\i\tensor\tau_{-i/2}$ to this operator gives $\sum_{\a,\be}((f_{\be}\tensor e_{\a})\tensor\Xi^{*})\tensor(\om_{(\one\tensor e_{\a}^{*})(\Xi),f_{\be}}\tensor\i)(V^{*})$.
The closedness of $\i\tensor\i\tensor\tau_{-i/2}$ thus yields the
desired conclusion.
\end{proof}

\subsection{Weak mixing}
\begin{defn}
\label{def:matched_vectors}Let $\H,\K$ be Hilbert spaces. Vectors
$\Xi,\Upsilon\in\H\tensor\K$ are said to be \emph{matched} if for
every $\z\in\H$ and $\xi\in\K$,
\begin{itemize}
\item $\left\langle (\z^{*}\tensor\one)\Xi,(\z^{*}\tensor\one)\Upsilon\right\rangle ,\left\langle (\one\tensor\xi^{*})\Xi,(\one\tensor\xi^{*})\Upsilon\right\rangle \in\R$,
and 
\item $(\z^{*}\tensor\one)\Xi=0\iff(\z^{*}\tensor\one)\Upsilon=0$, $(\one\tensor\xi^{*})\Xi=0\iff(\one\tensor\xi^{*})\Upsilon=0$.
\end{itemize}
\end{defn}
Let $\H,\K$ be Hilbert spaces, $\mathcal{J}:\H\to\J\H$ an anti-unitary
and $\Xi,\Upsilon\in\J\H\tensor\K$. Define $T_{\Xi},T_{\Upsilon}\in B(\H,\K)$
by $T_{\Xi}\z:=((\J\z)^{*}\tensor\one)(\Xi)$, $\z\in\H$, and similarly
for $T_{\Upsilon}$. Notice that $T_{\Xi}^{*}T_{\Upsilon}\in B(\H)$
and $T_{\Xi}T_{\Upsilon}^{*}\in B(\K)$ are selfadjoint if and only
if the first condition of \prettyref{def:matched_vectors} holds,
while the second is equivalent to $\ker T_{\Xi}=\ker T_{\Upsilon}$
and $\ker T_{\Xi}^{*}=\ker T_{\Upsilon}^{*}$. Hence, by \prettyref{lem:abc},
$\Xi,\Upsilon$ are matched if and only if there exists a (generally
unbounded) injective selfadjoint operator $B$ over $\K$ commuting
with $T_{\Xi}T_{\Xi}^{*}$ such that $T_{\Upsilon}=BT_{\Xi}$. The
latter just means that $\Upsilon=(\one\tensor B)\Xi$.

We use the convention that for $y\in\Linfty{\G}$ and $c\geq0$, we
write $\left\Vert \tau_{i/2}(y)\right\Vert >c$ to either mean that
$y$ does not belongs to $D(\tau_{i/2})$ or that it does and genuinely
$\left\Vert \tau_{i/2}(y)\right\Vert >c$.
\begin{defn}
\label{def:weak_mixing}Let $\G$ be a LCQG and let $U,V$ be unitary
co-representations of $\G$ on Hilbert spaces $\H,\K$, respectively.
We introduce the following \emph{weak mixing} conditions.
\begin{description}
\item [{(WM)\label{desc:weak_mixing}}] For every $\z_{1},\ldots,\z_{n}\in\H$
such that some element of $\mathcal{BR}(\overline{U}\tp U')$ contains
all vectors $\J\z_{i}\tensor\z_{i}$, $1\leq i\leq n$, and every
$\e>0$, \emph{either} there is a state $\om\in\Lone{\G}$ such that
$\om\left(\left|(\om_{\z_{i},\z_{j}}\tensor\i)(U)\right|\right)\leq\e$
for each $1\leq i,j\leq n$, \emph{or} there exist $1\leq i,j\leq n$
and $\rho_{1},\rho_{2}\in\Lone{\G}$ of norm $1$ such that  $\left\Vert \tau_{i/2}\left[\left(\om_{(\i\tensor\rho_{1})(U^{*})\z_{i},(\i\tensor\rho_{2})(U^{*})\z_{j}}\tensor\i\right)(U)\right]\right\Vert >\e^{-1}$.
\item [{(PSE)\label{desc:prod_sim_ergodic}}] The tensor product co-representations
$V\tp U$ and $V\tp U'$ are \emph{simultaneously ergodic}: they do
not admit nonzero \emph{matched} invariant vectors in $\K\tensor\H$.
\item [{(PEB)\label{desc:prod_ergodic_bdd}}] The tensor product co-representation
$V\tp U'$ does not admit a nonzero invariant vector $\Xi\in\K\tensor\H$
such that $(\om_{(\z^{*}\tensor\one)\Xi,(\xi^{*}\tensor\one)\Xi}\tensor\i)(U)\in D(\tau_{i/2})$
for all $\z,\xi\in\K$.
\item [{(NCAP)\label{desc:no_CAP_vect}}] There are no nonzero completely
(almost) periodic vectors with respect to $U$, that is, $U$ does
not admit a nonzero finite-dimensional sub-representation with an
invertible transpose.
\end{description}
If $\G$ is amenable and $m\in\Linfty{\G}^{*}$ is a right-invariant
mean for $\G$, we introduce the following strengthening of \prettyref{desc:weak_mixing}:
\begin{description}
\item [{(WMa)\label{desc:weak_mixing_amenable}}] For every $\z_{1},\ldots,\z_{n}\in\H$
such that some element of $\mathcal{BR}(\overline{U}\tp U')$ contains
all vectors $\J\z_{i}\tensor\z_{i}$, $1\leq i\leq n$, \emph{either}
for each $1\leq i,j\leq n$ we have $m\left(\left|(\om_{\z_{i},\z_{j}}\tensor\i)(U)\right|\right)=0$,
\emph{or} for every $\e>0$ there exist $1\leq i,j\leq n$ and $\rho_{1},\rho_{2}\in\Lone{\G}$
of norm $1$ such that $\left\Vert \tau_{i/2}\left[\left(\om_{(\i\tensor\rho_{1})(U^{*})\z_{i},(\i\tensor\rho_{2})(U^{*})\z_{j}}\tensor\i\right)(U)\right]\right\Vert >\e^{-1}$.
\end{description}
\end{defn}

\begin{defn}
\label{def:strict_weak_mixing}Let $\G$ be a LCQG and let $U,V$
be unitary co-representations of $\G$ on Hilbert spaces $\H,\K$,
respectively. We introduce the following \emph{strict weak mixing}
conditions.
\begin{description}
\item [{(sWM)\label{desc:strict_weak_mixing}}] For every $\z_{1},\ldots,\z_{n}\in\H$
such that some element of $\mathcal{BR}(\overline{U}\tp U')$ contains
all vectors $\J\z_{i}\tensor\z_{i}$, $1\leq i\leq n$, and every
$\e>0$, there is a state $\om\in\Lone{\G}$ such that $\om\left(\left|(\om_{\z_{i},\z_{j}}\tensor\i)(U)\right|\right)\leq\e$
for each $1\leq i,j\leq n$.
\item [{(PE)\label{desc:prod_ergodic}}] The tensor product co-representation
$V\tp U'$ is \emph{ergodic}.
\item [{(NFDS)\label{desc:no_FDS}}] There is no nonzero finite-dimensional
sub-representation of $U$.
\end{description}
If $\G$ is amenable and $m\in\Linfty{\G}^{*}$ is a right-invariant
mean for $\G$, define:
\begin{description}
\item [{(sWMa)\label{desc:strict_weak_mixing_amenable}}] For every $\z_{1},\ldots,\z_{n}\in\H$
such that some element of $\mathcal{BR}(\overline{U}\tp U')$ contains
all vectors $\J\z_{i}\tensor\z_{i}$, $1\leq i\leq n$, we have $m\left(\left|(\om_{\z_{i},\z_{j}}\tensor\i)(U)\right|\right)=0$
for each $1\leq i,j\leq n$.
\end{description}
\end{defn}
The following is the main result of this section.
\begin{thm}
\label{thm:wm_conditions}Let $\G$ be a LCQG and let $U$ be a unitary
co-representation of $\G$.
\begin{enumerate}
\item \label{enu:wm_conditions_1}The conditions (\prettyref{desc:prod_sim_ergodic}
for $V:=\overline{U}$), (\prettyref{desc:prod_sim_ergodic} holding
for every $V$), (\prettyref{desc:prod_ergodic_bdd} for $V:=\overline{U}$),
(\prettyref{desc:prod_ergodic_bdd} holding for every $V$) and \prettyref{desc:no_CAP_vect}
are equivalent, and are implied by \prettyref{desc:weak_mixing}.
\item \label{enu:wm_conditions_2}We have \prettyref{desc:strict_weak_mixing}$\implies$\prettyref{desc:no_FDS}$\implies$(\prettyref{desc:prod_ergodic}
holding for every $V$)$\implies$(\prettyref{desc:prod_ergodic}
for $V:=\overline{U}$).
\item If $\G$ is amenable, then all conditions in \prettyref{enu:wm_conditions_1}
are equivalent to one another and to \prettyref{desc:weak_mixing_amenable},
and all conditions in \prettyref{enu:wm_conditions_2} are equivalent
to one another and to \prettyref{desc:strict_weak_mixing_amenable}.
\item If the scaling group of $\G$ is trivial, then all conditions in \prettyref{enu:wm_conditions_1}
and in \prettyref{enu:wm_conditions_2} are equivalent to one another.
\end{enumerate}
\end{thm}
\begin{rem}

\begin{enumerate}
\item It is clear that \prettyref{desc:no_FDS}$\implies$\prettyref{desc:no_CAP_vect}$\implies$ergodicity.
When $\G$ is non-compact, mixing implies \prettyref{desc:no_FDS}:
every nontrivial finite-dimensional sub-representation has entries
in $\Cz{\G}$ by mixing, so its unitarity forces $\Cz{\G}$ to be
unital, a contradiction. When $\G$ is compact, every unitary co-representation
is trivially mixing, but never satisfies \prettyref{desc:no_CAP_vect}
by the general theory \citep{Woronowicz__symetries_quantiques}.
\item The advantage of \prettyref{desc:prod_sim_ergodic} over \prettyref{desc:prod_ergodic_bdd}
is that the former can be written without any reference to the scaling
group.
\item When $\G$ is a locally compact group, our definition of weak mixing
reduces to the classical one.
\item When $\G$ is discrete, \prettyref{desc:no_CAP_vect} is equivalent
to \prettyref{desc:no_FDS} by \prettyref{prop:DQG_invertible_transpose}.
\item Let $\HH$ be a closed quantum subgroup of $\G$ in the sense of Woronowicz
\citep[Definition 3.2 and Theorem 3.6]{Daws_Kasprzak_Skalski_Soltan__closed_q_subgroups_LCQGs},
and let $\pi:\CzU{\G}\to\CzU{\HH}$ be the associated surjective $*$-homomorphism.
It follows from \citep[Proposition 3.10 and Theorem 4.8]{Meyer_Roy_Woronowicz__hom_quant_grps}
that $R^{\mathrm{u},\HH}\circ\pi=\pi\circ R^{\mathrm{u},\G}$ and
$\tau_{t}^{\mathrm{u},\HH}\circ\pi=\pi\circ\tau_{t}^{\mathrm{u},\G}$
for every $t\in\R$. Assume that $\G,\HH$ are co-amenable. If $U$
is a unitary co-representation of $\G$ on $\H$, then its ``restriction
to $\HH$'' $U_{\HH}:=(\i\tensor\pi)(U)$ (view $\pi$ as a $\Cz{\G}\to\Cz{\HH}$
map) is a unitary co-representation of $\HH$ on $\H$. If either
of \prettyref{desc:no_FDS} or \prettyref{desc:no_CAP_vect} holds
for $U_{\HH}$, then it holds for $U$.
\end{enumerate}
\end{rem}
In the course of the proof of \prettyref{thm:wm_conditions}, which
is divided into several steps, we assume that $U$ is a unitary co-representation
on a Hilbert space $\H$ and fix an anti-unitary $\J:\H\to\J\H$.
Notice first that when $\G$ is amenable, \prettyref{desc:weak_mixing_amenable}$\implies$\prettyref{desc:weak_mixing}
and \prettyref{desc:strict_weak_mixing_amenable}$\implies$\prettyref{desc:strict_weak_mixing}.
Also, \prettyref{desc:no_FDS} and \prettyref{desc:no_CAP_vect} are
equivalent if the scaling group of $\G$ is trivial by \prettyref{sub:prelim_CAP}.

The following ``mean ergodic theorem'' is elementary.
\begin{lem}
\label{lem:mean_erg_thm}Let $V\in B(\H)\tensorn\Linfty{\HH}$ (resp.,
$V\in M(\mathbb{K}(\H)\tensormin C(\HH))$) be a co-representation
of an amenable LCQG (resp., a not necessarily reduced $C^{*}$-algebraic
compact quantum group \citep{Woronowicz__symetries_quantiques,Maes_van_Daele__notes_CQGs})
$\HH$ on a Hilbert space $\H$. Then $(\i\tensor m)(V)$ is an idempotent
whose image consists of all vectors invariant under $V$, where $m$
is a left-invariant mean for (resp., the Haar state of) $\HH$.\end{lem}
\begin{prop}
\prettyref{desc:weak_mixing} $\lor$ (\prettyref{desc:prod_sim_ergodic}
with $V:=\overline{U}$) $\lor$ (\prettyref{desc:prod_ergodic_bdd}
with $V:=\overline{U}$) $\implies$\prettyref{desc:no_CAP_vect},
and \prettyref{desc:strict_weak_mixing}$\implies$\prettyref{desc:no_FDS}.\end{prop}
\begin{proof}
Assume first that \prettyref{desc:no_FDS} does not hold, and let
$\H_{0}$ be a nonzero finite-dimensional subspace of $\H$ that is
invariant under $U$. Denote by $u\in B(\H_{0})\tensor\Linfty{\G}$
the ambient sub-representation of $U$. One checks that $\H_{0},\J\H_{0}$
are invariant under $U',\overline{U}$, respectively, where the invariance
under $U'$ means that $(\i\tensor\tau_{-i/2})(U(p_{\H_{0}}\tensor\one))$
commutes with $p_{\H_{0}}\tensor\one$. Let $u',\overline{u}$ be
the corresponding (bounded, the former generally not unitary) sub-representations.
We have $\J\H_{0}\tensor\H_{0}\in\mathcal{BR}(\overline{U}\tp U')$.
Let $\z_{1},\ldots,\z_{n}$ be an orthonormal basis of $\H_{0}$.
Write $u=\left(u_{ij}\right)_{i,j=1}^{n}$ with respect to this basis,
that is, $u_{ji}=(\om_{\z_{i},\z_{j}}\tensor\i)(U)$ for every $1\leq i,j\leq n$.
Since $u$ is unitary, for every state $\om\in\Lone{\G}$ there exist
$1\leq i,j\leq n$ such that $\om(\left|u_{ij}\right|)\geq\om(\left|u_{ij}\right|^{2})\geq\frac{1}{n}$.
Hence, \prettyref{desc:strict_weak_mixing} fails.

Suppose henceforth that \prettyref{desc:no_CAP_vect} does not hold.
So we can assume that $u$ has an invertible transpose, or equivalently,
that $u\in D(\i\tensor\tau_{i/2})$ (\prettyref{sub:prelim_CAP}).
As $u$ is finite dimensional, there is $M<\infty$ such that for
all $1\leq i,j\leq n$ and $\rho_{1},\rho_{2}\in\Lone{\G}$ of norm
$1$, $y:=(\om_{(\i\tensor\rho_{1})(U^{*})\z_{i},(\i\tensor\rho_{2})(U^{*})\z_{j}}\tensor\i)(U)=(\om_{(\i\tensor\rho_{1})(u^{*})\z_{i},(\i\tensor\rho_{2})(u^{*})\z_{j}}\tensor\i)(u)$
satisfies $y\in D(\tau_{i/2})$ and $\left\Vert \tau_{i/2}(y)\right\Vert \leq M$.
In conclusion, \prettyref{desc:weak_mixing} fails.

Denote by $\HH$ the $C^{*}$-algebraic compact quantum group induced
by $u$ as in \prettyref{sub:prelim_CAP} and by $h$ the Haar state
on $\HH$. Notice that $h$ is not necessarily faithful because $\HH$
is not necessarily reduced. Nonetheless, there is no ambiguity about
the scaling group and unitary antipode by \prettyref{lem:compactification_antipode}.
The components of $u$ belong to $C(\HH)$ by definition, so $u,u',\overline{u}$
are co-representations of $\HH$.

We may and do assume that $u$ is irreducible. Hence there exists
a strictly positive matrix $F\in M_{\dim\H_{0}}$, which we view as
an operator on $\H_{0}$, such that $(\i\tensor\tau_{z})(u)=(F^{-iz}\tensor\one)u(F^{iz}\tensor\one)$
for every $z\in\C$ \citep[Section 7]{Woronowicz__symetries_quantiques}. 

Since $\overline{u}_{13}u_{23}'$ is a co-representation of $\HH$
on $\J\H_{0}\tensor\H_{0}$, every vector in the image of $(\i\tensor\i\tensor h)(\overline{u}_{13}u_{23}')$
is invariant under $\overline{u}_{13}u_{23}'$ by \prettyref{lem:mean_erg_thm},
thus under $\overline{U}\tp U'$ (see \prettyref{lem:V_U'_invariant_vector}).
Fix $0\neq\z\in\H_{0}$, and consider the invariant vector $\Xi:=((\i\tensor\i\tensor h)(\overline{u}_{13}u_{23}'))(\J\z\tensor\z)$.
To show that it is nonzero, let $\eta\in\H_{0}$ be such that $x:=(\om_{\eta,\z}\tensor\i)(u)\neq0$.
From \prettyref{lem:U'_left_leg}, \prettyref{eq:U'_left_leg__1}
we obtain 
\[
\begin{split}\left\langle \Xi,\J\eta\tensor\eta\right\rangle  & =\left\langle ((\i\tensor\i\tensor h)(\overline{u}_{13}u_{23}'))(\J\z\tensor\z),\J\eta\tensor\eta\right\rangle \\
 & =h\left[(\om_{\J\z,\J\eta}\tensor\i)(\overline{u})(\om_{\z,\eta}\tensor\i)(u')\right]=h(R(x^{*}x))=h(x^{*}x).
\end{split}
\]
By \citep[Proposition 4.2]{Woronowicz__symetries_quantiques}, this
number is strictly positive. Therefore $\Xi\neq0$, and (\prettyref{desc:prod_ergodic_bdd}
with $V:=\overline{U}$) fails.

Similarly, $\overline{u}_{13}u_{23}$ is also a co-representation
of $\HH$ on $\J\H_{0}\tensor\H_{0}$. The vector $\Upsilon:=((\i\tensor\i\tensor h)(\overline{u}_{13}u_{23}))(\J\z\tensor F^{1/2}\z)$
is invariant under $\overline{u}_{13}u_{23}$, thus under $\overline{U}\tp U$,
by \prettyref{lem:mean_erg_thm}, and 
\[
\begin{split}(\i\tensor F^{-1/2})\Upsilon & =((\i\tensor\i\tensor h)(\overline{u}_{13}(\one\tensor F^{-1/2}\tensor\one)u_{23}(\one\tensor F^{1/2}\tensor\one)))(\J\z\tensor\z)\\
 & =((\i\tensor\i\tensor h)(\overline{u}_{13}u_{23}'))(\J\z\tensor\z)=\Xi.
\end{split}
\]
Defining $T_{\Xi}\in B(\H_{0})$ by $T_{\Xi}\xi:=((\J\xi)^{*}\tensor\i)(\Xi)$,
$\xi\in\H_{0}$, we get as above 
\[
\left\langle T_{\Xi}\xi,\xi'\right\rangle =h\left((\om_{\xi',\z}\tensor\i)(u)^{*}(\om_{\xi,\z}\tensor\i)(u)\right)
\]
for all $\xi,\xi'\in\H_{0}$. Thus 
\[
\begin{split}\left\langle T_{\Xi}F^{-1/2}\xi,F^{1/2}\xi'\right\rangle  & =h\left((\om_{\xi',\z}\tensor\i)(u(F^{1/2}\tensor\one))^{*}(\om_{\xi,\z}\tensor\i)(u(F^{-1/2}\tensor\one))\right)\\
 & =h\left((\om_{\xi',F^{1/2}\z}\tensor\i)((\i\tensor\tau_{-i/2})(u))^{*}(\om_{\xi,F^{-1/2}\z}\tensor\i)((\i\tensor\tau_{i/2})(u))\right).
\end{split}
\]
Taking the original $\z$ to be an eigenvector of $F$ and using that
$h$ is invariant under $\tau$, we conclude that $\left\langle T_{\Xi}F^{-1/2}\xi,F^{1/2}\xi'\right\rangle =\left\langle T_{\Xi}\xi,\xi'\right\rangle $,
namely that $T_{\Xi}$ and $F^{1/2}$ commute. Thus, defining $B\in B(\H)$
to be $F^{1/2}$ on $\H_{0}$ and the identity on $\H_{0}^{\perp}$,
we infer that $\Xi,\Upsilon$ are matched, and (\prettyref{desc:prod_sim_ergodic}
with $V:=\overline{U}$) fails.
\end{proof}
In the particular case of $\G$ with a trivial scaling group (where
$U'=U$), the next result was proved by Kyed and So{\ldash}tan \citep[Theorem 2.6]{Kyed_Soltan__prop_T_exotic_QG_norms},
Das and Daws \citep[Proposition 7.2]{Das_Daws__quantum_Eberlein}
and Chen and Ng \citep[Proposition 3.5]{Chen_Ng__prop_T_LCQGs}.
\begin{prop}
\label{prop:prod_CAP}For every unitary co-representation $V$ of
$\G$, \prettyref{desc:no_CAP_vect}$\implies$\prettyref{desc:prod_sim_ergodic},\prettyref{desc:prod_ergodic_bdd}
and \prettyref{desc:no_FDS}$\implies$\prettyref{desc:prod_ergodic}.\end{prop}
\begin{proof}
Let us replace $V$ by $\overline{V}$ for convenience. Fix an anti-unitary
$\J_{V}:\K\to\J_{V}\K$. Assume that \prettyref{desc:prod_ergodic}
does not hold, and let $\Xi\in\J_{V}\K\tensor\H$ be a nonzero invariant
vector for $\overline{V}\tp U'$ in the sense of \prettyref{def:BR}.
It induces a Hilbert--Schmidt operator $T_{\Xi}:\K\to\H$ given by
$T_{\Xi}\z:=((\J_{V}\z)^{*}\tensor\one)(\Xi)$, $\z\in\K$. For all
$\z\in\K$ and $\xi\in\H$, by invariance of $\Xi$ under $\overline{V}\tp U'$
(replacing $V$ by $\overline{V}$ in \prettyref{lem:V_U'_invariant_vector})
and \prettyref{lem:U'_left_leg}, \prettyref{eq:U'_left_leg__2} with
$V$ instead of $U$, 
\begin{equation}
(\om_{T_{\Xi}\z,\xi}\tensor\i)(U')=(\om_{(\one\tensor\xi^{*})(\Xi),\J_{V}\z}\tensor\i)(\overline{V}^{*})=(\om_{\z,T_{\Xi}^{*}\xi}\tensor\i)(V').\label{eq:prod_CAP__U'_V'}
\end{equation}
Formally, this means that $U'(T_{\Xi}\tensor\one)=(T_{\Xi}\tensor\one)V'$.
Since $U'=(\i\tensor\tau_{-i/2})(U)$ and the same for $V$, and since
$\tau_{-i/2}$ is injective, \prettyref{eq:prod_CAP__U'_V'} holds
with $U,V$ in place of $U',V'$, respectively, thus $U(T_{\Xi}\tensor\one)=(T_{\Xi}\tensor\one)V$.
Hence, the compact, nonzero, positive operator $T_{\Xi}T_{\Xi}^{*}$
over $\H$ intertwines $U$ with itself. Fixing a strictly positive
eigenvalue, the associated finite-dimensional spectral subspace $\H_{0}$
of $T_{\Xi}T_{\Xi}^{*}$ is invariant under $U$, so that \prettyref{desc:no_FDS}
does not hold. 

We should now show that under additional assumptions, the restriction
$u$ of $U$ to $\H_{0}$ has an invertible transpose. If \prettyref{desc:prod_sim_ergodic}
fails, then there is $\Upsilon\in\J_{V}\K\tensor\H$ invariant under
$\overline{V}\tp U$ such that $\Xi,\Upsilon$ are matched. As above,
the induced Hilbert--Schmidt operator $T_{\Upsilon}:\K\to\H$ given
by $T_{\Upsilon}\z:=((\J_{V}\z)^{*}\tensor\one)(\Upsilon)$, $\z\in\K$,
satisfies, for all $\z\in\K$ and $\xi\in\H$,
\begin{equation}
(\om_{T_{\Upsilon}\z,\xi}\tensor\i)(U)=(\om_{(\one\tensor\xi^{*})(\Upsilon),\J_{V}\z}\tensor\i)(\overline{V}^{*})=(\om_{\z,T_{\Upsilon}^{*}\xi}\tensor\i)(V')\label{eq:U__V_tag}
\end{equation}
using \prettyref{lem:U'_left_leg}, \prettyref{eq:U'_left_leg__2}
with $V$ instead of $U$. Thus, formally, $U(T_{\Upsilon}\tensor\one)=(T_{\Upsilon}\tensor\one)V'$.
As $\Xi,\Upsilon$ are matched, there is, by \prettyref{lem:abc},
a generally unbounded, injective, selfadjoint operator $B$ over $\H$
commuting with $T_{\Xi}T_{\Xi}^{*}$ such that $T_{\Upsilon}=BT_{\Xi}$
(see the paragraph succeeding \prettyref{def:matched_vectors}). Consequently,
$B$ maps $\H_{0}$ onto itself. If now $\xi_{1},\xi_{2}\in\H_{0}$
and $\z\in\K$ is such that $T_{\Xi}\z=B^{-1}\xi_{1}$, then from
\prettyref{eq:U__V_tag} we deduce that 
\[
(\om_{\xi_{1},\xi_{2}}\tensor\i)(U)=(\om_{\z,T_{\Upsilon}^{*}\xi_{2}}\tensor\i)(V')=\tau_{-i/2}[(\om_{\z,T_{\Upsilon}^{*}\xi_{2}}\tensor\i)(V)]\in D(\tau_{i/2}).
\]
Consequently, $u\in D(\i\tensor\tau_{i/2})$, that is, $u^{\t}$ is
invertible (see \prettyref{sub:prelim_CAP}).

If \prettyref{desc:prod_ergodic_bdd} fails, then we may assume that
$(\om_{T_{\Xi}\z,T_{\Xi}\xi}\tensor\i)(U)\in D(\tau_{i/2})$ for all
$\z,\xi\in\K$. In particular, $u\in D(\i\tensor\tau_{i/2})$ again.
This completes the proof.\end{proof}
\begin{rem}
\label{rem:calc_Z}In the next proof we use the following simple observation.
Suppose that $\mathcal{F}\in\mathcal{BR}(\overline{U}\tp U')$ and
$\z\in\H$ is such that $\J\z\tensor\z\in\mathcal{F}$. Write $p_{\mathcal{F}}$
for the projection of $\J\H\tensor\H$ onto $\mathcal{F}$, $p_{\C\z}$
for the projection of $\H$ onto $\C\z$ and $Z:=\overline{U}_{13}(\i\tensor\i\tensor\tau_{-i/2})(U_{23}(p_{\mathcal{F}}\tensor\one))$.
Then $(\i\tensor\i\tensor\tau_{-i/2})(U_{23}(p_{\mathcal{F}}\tensor\one))(\J\z\tensor\z\tensor\eta)=\J\z\tensor(\i\tensor\tau_{-i/2})(U(p_{\C\z}\tensor\one))(\z\tensor\eta)$
for all $\eta\in\Ltwo{\G}$, and therefore, for every $\xi,\xi'\in\H$
and $\om\in\Lone{\G}$, we have 
\[
\begin{split}(\om_{\J\z,\J\xi}\tensor\om_{\z,\xi'}\tensor\om)(Z) & =\om\left[(\om_{\J\z,\J\xi}\tensor\i)(\overline{U})\tau_{-i/2}\left((\om_{\z,\xi'}\tensor\i)(U)\right)\right]\\
 & =\om\left[(\om_{\J\z,\J\xi}\tensor\i)(\overline{U})(\om_{\z,\xi'}\tensor\i)(U')\right]
\end{split}
\]
by the definition of $U'$.\end{rem}
\begin{prop}

\begin{enumerate}
\item Assume that $\G$ is amenable. Then the condition (\prettyref{desc:prod_ergodic_bdd}
with $V:=\overline{U}$), resp.~(\prettyref{desc:prod_ergodic} with
$V:=\overline{U}$), implies \prettyref{desc:weak_mixing_amenable},
resp.~\prettyref{desc:strict_weak_mixing_amenable}. 
\item Assume that the scaling group of $\G$ is trivial. Then the condition
(\prettyref{desc:prod_ergodic} with $V:=\overline{U}$) implies \prettyref{desc:strict_weak_mixing}. 
\end{enumerate}
\end{prop}
\begin{proof}
All implications are proved by contraposition. Hence, we are given
$\z_{1},\ldots,\z_{n}\in\H$ without the respective property. Out
of them we obtain a vector $\Xi$ invariant under $\overline{U}\tp U'$
by averaging, as follows. Let $\Theta:=\sum_{i=1}^{n}\J\z_{i}\tensor\z_{i}$. 

Case I: $\G$ is amenable and \prettyref{desc:strict_weak_mixing_amenable}
does not hold. Letting $\mathcal{F}$ be the given element of $\mathcal{BR}(\overline{U}\tp U')$,
we have $\J\z_{i}\tensor\z_{i}\in\mathcal{F}$ for every $1\leq i\leq n$,
so $\Theta\in\mathcal{F}$. Let $p_{\mathcal{F}}$ be the projection
of $\J\H\tensor\H$ onto $\mathcal{F}$ and $Z:=\overline{U}_{13}(\i\tensor\i\tensor\tau_{-i/2})(U_{23}(p_{\mathcal{F}}\tensor\one))$.
Then $Z$ is a (bounded, generally not unitary) co-representation
of $\G$ on $\J\H\tensor\H$ commuting with $p_{\mathcal{F}}\tensor\one$.
If $m\in\Linfty{\G}^{*}$ is the given right-invariant mean for $\G$,
then $m':=m\circ R$ is a left-invariant mean for $\G$ by \prettyref{eq:R_Delta}.
Let $\Xi:=((\i\tensor\i\tensor m')(Z))\Theta\in\mathcal{F}$. By \prettyref{lem:mean_erg_thm},
$\Xi$ is invariant under $Z$, thus under $\overline{U}\tp U'$ in
the sense of \prettyref{def:BR}.

Case II: the scaling group is trivial (so $U'=U$) and \prettyref{desc:strict_weak_mixing}
does not hold. Consider the closed convex set 
\[
K:=\overline{\left\{ ((\i\tensor\i\tensor\om)(\overline{U}\tp U))\Theta:\om\text{ is a state in }\Lone{\G}\right\} }.
\]
Since $\overline{U}\tp U$ is a unitary co-representation of $\G$,
$K$ is invariant under $\overline{U}\tp U$, because if $\om_{1},\om_{2}\in\Lone{\G}$
are states, then $(\i\tensor\i\tensor\om_{1})(\overline{U}\tp U)(\i\tensor\i\tensor\om_{2})(\overline{U}\tp U)=(\i\tensor\i\tensor(\om_{1}*\om_{2}))(\overline{U}\tp U)$.
Let $\Xi$ be the unique element of minimal norm in $K$. Then $\Xi$
is invariant under $\overline{U}\tp U$ because for every state $\om$
in $\Lone{\G}$, $((\i\tensor\i\tensor\om)(\overline{U}\tp U))\Xi\in K$
has norm at most $\left\Vert \Xi\right\Vert $. Let $Z:=\overline{U}\tp U$.

Treating both cases together, we show that $\Xi\neq0$. Let $\om$
be a state of $\Linfty{\G}$. By \prettyref{rem:calc_Z} and \prettyref{lem:U'_left_leg},
\prettyref{eq:U'_left_leg__1}, for every $\z\in\left\{ \z_{1},\ldots,\z_{n}\right\} $
and $\xi\in\H$, 
\[
\begin{split}\left\langle ((\i\tensor\i\tensor\om)(Z))(\J\z\tensor\z),\J\xi\tensor\xi\right\rangle  & =\om\left[(\om_{\J\z,\J\xi}\tensor\i)(\overline{U})(\om_{\z,\xi}\tensor\i)(U')\right]\\
 & =(\om\circ R)\left[(\om_{\xi,\z}\tensor\i)(U)^{*}(\om_{\xi,\z}\tensor\i)(U)\right].
\end{split}
\]
Writing $\Xi_{\om}:=((\i\tensor\i\tensor\om)(Z))\Theta$, we infer
that
\begin{equation}
\left\langle \Xi_{\om},\Theta\right\rangle =\sum_{i,j=1}^{n}(\om\circ R)\bigl(\bigl|(\om_{\z_{j},\z_{i}}\tensor\i)(U)\bigr|^{2}\bigr).\label{eq:Xi_om__Theta}
\end{equation}
By assumption, in case I, there exist $i,j$ such that $m(|(\om_{\z_{j},\z_{i}}\tensor\i)(U)|^{2})>0$,
hence $\left\langle \Xi,\Theta\right\rangle >0$ by setting $\om:=m'$
in \prettyref{eq:Xi_om__Theta}. In case II, there is $\e_{0}>0$
such that for every state $\om$ in $\Lone{\G}$ there exist $i,j$
such that $(\om\circ R)(|(\om_{\z_{j},\z_{i}}\tensor\i)(U)|^{2})\geq\e_{0}^{2}$,
hence $\left\langle \Xi_{\om},\Theta\right\rangle \geq\e_{0}^{2}$
by \prettyref{eq:Xi_om__Theta}. Therefore $\Xi\neq0$ either way,
failing \prettyref{desc:prod_ergodic} with $\overline{U}$ for $V$.

Now suppose that $\G$ is amenable and \prettyref{desc:weak_mixing_amenable}
does not hold (with respect to the same $\z_{1},\ldots,\z_{n}\in\H$).
Fix $\z,\z'\in\H$. To fail \prettyref{desc:prod_ergodic_bdd} with
$\overline{U}$ for $V$, we need to establish that $(\om_{T_{\Xi}\z,T_{\Xi}\z'}\tensor\i)(U)\in D(\tau_{i/2})$.
Let $\om,\om'$ be states in $\Lone{\G}$. Write $\rho_{i}:=\om\cdot(\om_{\J\z_{i},\J\z}\tensor\i)(\overline{U})\in\Lone{\G}$,
$1\leq i\leq n$. From \prettyref{rem:calc_Z} we get
\[
\begin{split}T_{\Xi_{\om}}\z & =((\J\z)^{*}\tensor\i)(\Xi_{\om})=\sum_{i=1}^{n}((\J\z)^{*}\tensor\one)\left[((\i\tensor\i\tensor\om)(Z))(\J\z_{i}\tensor\z_{i})\right]\\
 & =\sum_{i=1}^{n}((\i\tensor\rho_{i})(U'))\z_{i}=\sum_{i=1}^{n}(\i\tensor(\rho_{i}\circ R))(U^{*})\z_{i}.
\end{split}
\]
Similarly, $T_{\Xi_{\om'}}\z'=\sum_{j=1}^{n}(\i\tensor(\rho_{j}'\circ R))(U^{*})\z_{j}$
with suitable $\rho_{j}'\in\Lone{\G}$, $1\leq j\leq n$. Let 
\[
a_{\om,\om'}:=(\om_{T_{\Xi_{\om}}\z,T_{\Xi_{\om'}}\z'}\tensor\i)(U)=\sum_{i,j=1}^{n}(\om_{(\i\tensor(\rho_{i}\circ R))(U^{*})\z_{i},(\i\tensor(\rho_{j}'\circ R))(U^{*})\z_{j}}\tensor\i)(U).
\]
By assumption, $a_{\om,\om'}$ belongs to $D(\tau_{i/2})$ and satisfies
\begin{equation}
\left\Vert \tau_{i/2}(a_{\om,\om'})\right\Vert \leq\e_{0}^{-1}\sum_{i,j=1}^{n}\left\Vert \rho_{i}\circ R\right\Vert \left\Vert \rho_{j}'\circ R\right\Vert \leq\e_{0}^{-1}\sum_{i,j=1}^{n}\left\Vert \z_{i}\right\Vert \left\Vert \z_{j}\right\Vert \left\Vert \z\right\Vert \left\Vert \z'\right\Vert .\label{eq:tau_i_2__unif_bdd}
\end{equation}

Choose a net $\left(\om_{\iota}\right)$ of states in $\Lone{\G}$
that is $w^{*}$-convergent to $m'$. Hence $\Xi_{\om_{\iota}}\to\Xi$
weakly. By \prettyref{eq:tau_i_2__unif_bdd}, $a_{\om_{\iota},\om_{k}}$
belongs to $D(\tau_{i/2})$ for every $\iota,\kappa$ and $(\Vert\tau_{i/2}(a_{\om_{\iota},\om_{k}})\Vert)_{\iota,\kappa}$
is bounded. When $\kappa$ is fixed, $a_{\om_{\iota},\om_{k}}\to a_{\om_{\kappa}}:=(\om_{T_{\Xi}\z,T_{\Xi_{\om_{\kappa}}}\z'}\tensor\i)(U)$
weakly. From \prettyref{lem:action_conv_bdd}, $a_{\om_{\kappa}}\in D(\tau_{i/2})$
and $(\tau_{i/2}(a_{\om_{\kappa}}))_{\kappa}$ is bounded. Now $a_{\om_{\kappa}}\to(\om_{T_{\Xi}\z,T_{\Xi}\z'}\tensor\i)(U)$
weakly, and \prettyref{lem:action_conv_bdd} is used again to conclude
that $(\om_{T_{\Xi}\z,T_{\Xi}\z'}\tensor\i)(U)\in D(\tau_{i/2})$,
as desired.
\end{proof}
This completes the proof of \prettyref{thm:wm_conditions}.
\begin{cor}
Let $\G$ be a LCQG with trivial scaling group and let $U$ be a unitary
co-representation of $\G$. If $U$ is weakly mixing, then so is $U^{\tpsmall n}$
for every $n>1$. \end{cor}
\begin{proof}
This follows easily from \prettyref{desc:prod_ergodic} with $V:=\overline{U}$
and with arbitrary $V$ being equivalent. 
\end{proof}

\section{Applications}

\subsection{\label{sub:Jacobs_de-Leeuw_Glicksberg}The noncommutative Jacobs--de
Leeuw--Glicksberg splitting theorem}

Suppose that a LCQG $\G$ acts on a von Neumann algebra $N$ via an
action $\a:N\to N\tensorn\Linfty{\G}$ that preserves a faithful normal
state $\theta$ of $N$. The \emph{Koopman co-representation} of the
dynamical system $(N,\theta,\G,\a)$ is the unitary implementation
$U\in B(\Ltwo{N,\theta})\tensorn\Linfty{\G}$ of $\a$, given by 
\begin{align}
(\om_{\Lambda_{\theta}(a),\Lambda_{\theta}(b)}\tensor\i)(U) & =(\theta\tensor\i)\left((b^{*}\tensor\one)\a(a)\right),\qquad\text{or equivalently }\label{eq:state_preserve_act_unitary_imp}\\
(\om_{\Lambda_{\theta}(a),\Lambda_{\theta}(b)}\tensor\i)(U^{*}) & =(\theta\tensor\i)\left(\a(b^{*})(a\tensor\one)\right),\label{eq:state_preserve_act_unitary_imp_adjoint}
\end{align}
for all $a,b\in N$. Since $\Lambda_{\theta}(\one)$ is invariant
under $U$, so is $\Ltwozero{N,\theta}:=\Ltwo{N,\theta}\ominus\C\Lambda_{\theta}(\one)$,
and we can apply the above results to the restriction of $U$ to $\Ltwozero{N,\theta}$.
When this restriction is (weakly) mixing, the action $\a$ is said
to be (weakly) mixing. From \prettyref{lem:action_with_invariant_state_J}
we obtain $\overline{U}=U$ when using $J_{\theta}$ as $\J$ (meaning
that as in the classical setting, ``$U$ is induced by the orthogonal
co-representation of the dynamical system''). This simplifies a little
most of the conditions for weak mixing. 

The next theorem is a fundamental consequence of \prettyref{thm:wm_conditions}.
Runde and the author generalized in \citep{Runde_Viselter_LCQGs_Ergodic_Thy}
the Jacobs--de Leeuw--Glicksberg splitting theorem of \citep{Niculescu_Stroh_Zsido__noncmt_recur,Zsido__splitting_noncomm_dyn_sys}
to state-preserving actions of quantum semigroups on von Neumann algebras.
We introduced the notion of \emph{completely almost periodic operators}
\citep[Definition 3.7 and Theorem 4.5]{Runde_Viselter_LCQGs_Ergodic_Thy},
and established, under suitable assumptions, the existence of a conditional
expectation $E^{\CAP}$ from the von Neumann algebra that was acted
on onto its subalgebra consisting of these operators. The kernel of
$E^{\CAP}$ was conjectured to have a weakly mixing nature, but how
to put this in exact terms was left open \citep[Corollary 4.10 and the preceding paragraph]{Runde_Viselter_LCQGs_Ergodic_Thy}.
This issue is settled in the following result, which is also the complement
of \citep[Corollary 3.14]{Runde_Viselter_LCQGs_Ergodic_Thy} relating
complete almost periodicity to recurrence.

Recall that actions of co-amenable LCQGs are assumed to satisfy \prettyref{eq:action_co_amen_LCQG}.
\begin{thm}
\label{thm:Jacobs_de-Leeuw_Glicksberg}Let $\G$ be a co-amenable,
amenable LCQG that acts on a von Neumann algebra $N$ via an action
$\a:N\to N\tensorn\Linfty{\G}$ that preserves a faithful normal state
$\theta$ of $N$. Denote by $U$ the unitary implementation of $\a$
and by $N^{\CAP}$ the von Neumann subalgebra of $N$ consisting of
all completely almost periodic operators with respect to $\a$, and
let $m\in\Linfty{\G}^{*}$ be a left-invariant mean for $\G$. Then
the unique $\theta$-preserving conditional expectation $E^{\CAP}$
from $N$ onto $N^{\CAP}$ has the following property:

if $a_{1},\ldots,a_{n}\in\ker E^{\CAP}$ are such that some element
of $\mathcal{BR}(U\tp U')$ contains all vectors $\Lambda_{\theta}(a_{i})\tensor J_{\theta}\Lambda_{\theta}(a_{i})$,
$1\leq i\leq n$, then \emph{either} for each $1\leq i,j\leq n$ we
have $m\bigl(\left|(\theta\tensor\i)((a_{j}^{*}\tensor\one)\a(a_{i}))\right|_{\mathrm{r}}\bigr)=0$,
\emph{or} for every $\e>0$ there exist $1\leq i,j\leq n$ and $\rho_{1},\rho_{2}\in\Lone{\G}$
of norm $1$ such that $y:=(\theta\tensor\i)\bigl[\bigl((\i\tensor\rho_{2})\a(a_{j})^{*}\tensor\one\bigr)\a\bigl((\i\tensor\rho_{1})\a(a_{i})\bigr)\bigr]$
satisfies $\left\Vert \tau_{i/2}(y)\right\Vert >\e^{-1}$.\end{thm}
\begin{proof}
The novelty is the weak mixing property of $\ker E^{\CAP}$, the rest
being a special case of \citep[Corollary 4.10]{Runde_Viselter_LCQGs_Ergodic_Thy}.
The set $\Ltwo{N,\theta}^{\CAP}$ of all completely almost periodic
vectors with respect to $\a$ is a closed subspace of $\Ltwo{N,\theta}$
invariant under the unitary implementation $U$ of $\a$ \citep[Definition 3.7]{Runde_Viselter_LCQGs_Ergodic_Thy}.
The restriction of $U$ to $\Ltwo{N,\theta}\ominus\Ltwo{N,\theta}^{\CAP}$
satisfies \prettyref{desc:no_CAP_vect} by definition, and so \prettyref{thm:wm_conditions}
applies, and we infer that \prettyref{desc:weak_mixing_amenable}
holds. Observe that each $a\in N$ belongs to $\ker E^{\CAP}$ if
and only if $\Lambda_{\theta}(a)\in\Ltwo{N,\theta}\ominus\Ltwo{N,\theta}^{\CAP}$
because $\Lambda_{\theta}$ intertwines $E^{\CAP}$ and the projection
of $\Ltwo{N,\theta}$ onto $\Ltwo{N,\theta}^{\CAP}$. Moreover, $\Ltwo{N,\theta}\ominus\Ltwo{N,\theta}^{\CAP}$
is invariant under $J_{\theta}$. To get the desired result, take
$\left(J_{\theta}\Lambda_{\theta}(a_{i})\right)_{i=1}^{n}$ for $\left(\z_{i}\right)_{i=1}^{n}$
in \prettyref{desc:weak_mixing_amenable}, use \prettyref{eq:state_preserve_act_unitary_imp}
and the formulas 
\[
\begin{split}\begin{gathered}(\om_{J_{\theta}\z,J_{\theta}\eta}\tensor\i)(U)=R((\om_{\eta,\z}\tensor\i)(U)),\\
J_{\theta}(\i\tensor\overline{\rho\circ R})(U^{*})J_{\theta}=(\i\tensor\rho)(U)
\end{gathered}
 & \qquad(\forall\z,\eta\in\Ltwo{N,\theta},\rho\in\Lone{\G})\end{split}
\]
derived from \prettyref{lem:action_with_invariant_state_J} to get
$\left|(\om_{J_{\theta}\Lambda_{\theta}(a_{j}),J_{\theta}\Lambda_{\theta}(a_{i})}\tensor\i)(U)\right|=R\bigl(\left|(\theta\tensor\i)((a_{j}^{*}\tensor\one)\a(a_{i}))\right|_{\mathrm{r}}\bigr)$
and
\begin{multline*}
\bigl(\om_{(\i\tensor\overline{\rho_{2}\circ R})(U^{*})J_{\theta}\Lambda_{\theta}(a_{j}),(\i\tensor\overline{\rho_{1}\circ R})(U^{*})J_{\theta}\Lambda_{\theta}(a_{i})}\tensor\i\bigr)(U)=R\bigl(\bigl(\om_{(\i\tensor\rho_{1})(U)\Lambda_{\theta}(a_{i}),(\i\tensor\rho_{2})(U)\Lambda_{\theta}(a_{j})}\tensor\i\bigr)(U)\bigr)\\
=R\bigl\{(\theta\tensor\i)\bigl[\bigl((\i\tensor\rho_{2})\a(a_{j})^{*}\tensor\one\bigr)\a\bigl((\i\tensor\rho_{1})\a(a_{i})\bigr)\bigr]\bigr\}.
\end{multline*}
Then notice that $m\circ R$ is right invariant, and use the commutativity
of $\tau$ and $R$.
\end{proof}

\subsection{\label{sub:mixing_in_crossed_products}Relation to (weak) mixing
of inclusions in crossed products by discrete quantum groups}

A notion that is tightly related to weak mixing is the \emph{weak
asymptotic homomorphism property} of a masa in a II$_{1}$-factor,
which turned out to be equivalent to (strong) singularity of the masa,
and provided many examples of such masas (see Sinclair and Smith \citep{Sinclair_Smith__strong_singular},
Robertson, Sinclair and Smith \citep{Robertson_Sinclair_Smith__strong_sing}
and Sinclair, Smith, White and Wiggins \citep{Sinclair_Smith_White_Wiggins__strong_singular_MASAs}).
This was generalized and studied further by Jolissaint and Stalder
\citep{Jolissaint_Stalder__str_singular_MASAs} and by Cameron, Fang
and Mukherjee \citep{Cameron_Fang__Mukherjee__mix_subalg}. In particular,
they proved that in certain cases, when a discrete group $G$ acts
on a finite von Neumann algebra $N$ and preserves a trace, the inclusion
of $\VN G$ in the crossed product $N\rtimes G$ is (weakly) mixing
in appropriate senses if and only if the action is (weakly) mixing
(\citep[Propositions 2.2 and 3.6]{Jolissaint_Stalder__str_singular_MASAs}
and \citep[Proposition 1.1]{Cameron_Fang__Mukherjee__mix_subalg}).
In this subsection we show that, in the very general setting of discrete
quantum group actions, one direction of each of these implications
holds with respect to our definition of weakly mixing actions and
the definition of mixing in \citep{Daws_Fima_Skalski_White_Haagerup_LCQG}.
First, we extend the notion of (weakly) mixing inclusions of von Neumann
algebras \citep[p.~344]{Cameron_Fang__Mukherjee__mix_subalg} beyond
the finite case. 
\begin{defn}
Let $A\subseteq B$ be an inclusion of von Neumann algebras with a
faithful normal conditional expectation $E$ from $B$ onto $A$. 
\begin{itemize}
\item The inclusion $A\subseteq B$ is called \emph{$E$-weakly mixing}
if for every finite subset $F\subseteq B$ there exists a sequence
$\left(v_{n}\right)_{n=1}^{\infty}$ of unitaries in $A$ such that
\[
(\forall x,y\in F)\qquad E(xv_{n}y)-E(x)v_{n}E(y)\xrightarrow[n\to\infty]{}0\quad\text{strongly}.
\]

\item The inclusion $A\subseteq B$ is called \emph{$E$-mixing} if there
exists a subspace $C$ of $B$, dense in the bounded $*$-strong topology,
such that for every sequence $\left(v_{n}\right)_{n=1}^{\infty}$
of unitaries in $A$ converging weakly to zero, we have 
\begin{equation}
(\forall x\in C,y\in B)\qquad E(xv_{n}y)-E(x)v_{n}E(y)\xrightarrow[n\to\infty]{}0\quad\text{strongly}.\label{eq:mixing_inclusion}
\end{equation}

\end{itemize}
\end{defn}
\begin{rem}
If $B$ is a finite von Neumann algebra, so that the adjoint map is
strongly continuous on bounded sets, \prettyref{eq:mixing_inclusion}
holds as it is if and only if it holds for all $x,y\in B$.
\end{rem}
Let $\G$ be a discrete quantum group acting on a von Neumann algebra
$N$ via an action $\a$ that preserves a faithful normal state $\theta$
of $N$. Recall that the \emph{crossed product} $N\rtimes_{\a}\G$
\citep[pp.~434--435]{Vaes__unit_impl_LCQG} is the von Neumann subalgebra
of $N\tensorn B(\ltwo{\G})$ generated by $\a(N)$ and $\C\one\tensor\Linfty{\hat{\G}}'$.
We view $\Linfty{\hat{\G}}'$ as embedded in $N\rtimes_{\a}\G$. Since
$C:=\linspan\{\left(\one\tensor\hat{x}'\right)\a(a):a\in N,\hat{x}'\in\Linfty{\hat{\G}}'\}$
is dense in $N\rtimes_{\a}\G$ in the bounded $*$-strong topology
\citep[Lemma 3.3 and its proof]{Vaes__unit_impl_LCQG} and $(\theta\tensor\i_{B(\ltwo{\G})})\left[(\one\tensor\hat{x}')\a(a)\right]=\theta(a)\hat{x}'$
for every $a\in N,\hat{x}'\in\Linfty{\hat{\G}}'$, the map $E_{\Linfty{\hat{\G}}'}:=\theta\tensor\i_{B(\ltwo{\G})}$
is a faithful normal conditional expectation from $N\rtimes_{\a}\G$
onto $\Linfty{\hat{\G}}'$. Denoting by $\tilde{\theta}$ the dual
weight of $\theta$ \citep[Definition 3.1]{Vaes__unit_impl_LCQG},
this n.s.f.~weight on $N\rtimes_{\a}\G$ is actually a state as $\theta$
is a state and $\G$ is discrete, and $E_{\Linfty{\hat{\G}}'}$ is
just the unique $\tilde{\theta}$-preserving conditional expectation
from $N\rtimes_{\a}\G$ onto $\Linfty{\hat{\G}}'$. 
\begin{rem}
We deal with $\Linfty{\hat{\G}}'$ rather than $\Linfty{\hat{\G}}$
because $\a$ is a \emph{right} action. This is a mere technical matter.\end{rem}
\begin{prop}
In the above setting we have the following implications:
\begin{enumerate}
\item \label{enu:weak_mixing_in_crossed_products}if the inclusion $\Linfty{\hat{\G}}'\subseteq N\rtimes_{\a}\G$
is $E_{\Linfty{\hat{\G}}'}$-weakly mixing, then $\a$ satisfies \prettyref{desc:no_FDS};
\item \label{enu:mixing_in_crossed_products}if $\a$ is mixing, then the
inclusion $\Linfty{\hat{\G}}'\subseteq N\rtimes_{\a}\G$ is $E_{\Linfty{\hat{\G}}'}$-mixing.
\end{enumerate}
\end{prop}
\begin{proof}
Recall that $\Lambda_{\hat{\varphi}}(\one)=\Lambda_{\varphi}(p)$,
and thus, for every $b\in N$, \prettyref{eq:action_discrete_QG}
implies that $\a(b)(\Lambda_{\theta}(\one)\tensor\Lambda_{\hat{\varphi}}(\one))=(\Lambda_{\theta}\tensor\Lambda_{\varphi})(\a(b)(\one\tensor p))=(\Lambda_{\theta}\tensor\Lambda_{\varphi})(b\tensor p)=(b\tensor\one)(\Lambda_{\theta}(\one)\tensor\Lambda_{\hat{\varphi}}(\one))$.
Hence, for $a,b\in N$ and $\hat{x}\in\Linfty{\hat{\G}}$, 
\begin{equation}
\begin{split}E_{\Linfty{\hat{\G}}'}\bigl[\a(a)(\one\tensor\hat{J}\hat{x}\hat{J})\a(b)\bigr]\Lambda_{\hat{\varphi}}(\one) & =(\theta\tensor\i_{B(\ltwo{\G})})\bigl[\a(a)(\one\tensor\hat{J}\hat{x}\hat{J})\a(b)\bigr]\Lambda_{\hat{\varphi}}(\one)\\
 & =(\theta\tensor\i_{B(\ltwo{\G})})\bigl[\a(a)(\one\tensor\hat{J}\hat{x}\hat{J})(b\tensor\one)\bigr]\Lambda_{\hat{\varphi}}(\one)\\
 & =(\theta\tensor\i_{B(\ltwo{\G})})\left[\a(a)(b\tensor\one)\right]\hat{J}\Lambda_{\hat{\varphi}}(\hat{x}).
\end{split}
\label{eq:disc_crossed_prod_cond_exp}
\end{equation}

\prettyref{enu:weak_mixing_in_crossed_products} Suppose that $\a$
fails \prettyref{desc:no_FDS}. Denote by $U$ the unitary implementation
of $\a$, and let $u$ be a finite-dimensional sub-representation
of $U$ on a subspace of $\Ltwo{N,\theta}\ominus\C\Lambda_{\theta}(\one)$
spanned by the orthonormal set $\left\{ \z_{1},\ldots,\z_{n}\right\} $.
Write $u_{ij}:=(\om_{\z_{j},\z_{i}}\tensor\i)(U)$ ($1\leq i,j\leq n$).
Let $\e>0$. Pick $a_{1},\ldots,a_{n}\in N$ such that $\left\Vert \z_{i}-\Lambda_{\theta}(a_{i})\right\Vert <\e$
and $\theta(a_{i})=0$ for $1\leq i\leq n$. Then for every $\hat{v}\in\Linfty{\hat{\G}}$
of norm at most $1$ and $1\leq i,j\leq n$, we have from \prettyref{eq:disc_crossed_prod_cond_exp}
and \prettyref{eq:state_preserve_act_unitary_imp_adjoint}
\[
\begin{split}E_{\Linfty{\hat{\G}}'}\bigl[\a(a_{i}^{*})(\one\tensor\hat{J}\hat{v}\hat{J})\a(a_{j})\bigr]\Lambda_{\hat{\varphi}}(\one) & =(\om_{\Lambda_{\theta}(a_{j}),\Lambda_{\theta}(a_{i})}\tensor\i)(U^{*})\hat{J}\Lambda_{\hat{\varphi}}(\hat{v})\\
 & \approx_{(2+\e)\e}(\om_{\z_{j},\z_{i}}\tensor\i)(U^{*})\hat{J}\Lambda_{\hat{\varphi}}(\hat{v})=u_{ji}^{*}\hat{J}\Lambda_{\hat{\varphi}}(\hat{v}).
\end{split}
\]
Fix $1\leq j\leq n$. Applying $u_{ji}$ to the right-hand side and
summing for $i=1,\ldots,n$, we get $\hat{J}\Lambda_{\hat{\varphi}}(\hat{v})$
by the unitarity of $u$. Thus, if $\left(\hat{v}_{m}\right)$ are
unitaries in $\Linfty{\hat{\G}}$ with $E_{\Linfty{\hat{\G}}'}\bigl[\a(a_{i}^{*})(\one\tensor\hat{J}\hat{v}_{m}\hat{J})\a(a_{j})\bigr]\xrightarrow[m\to\infty]{}0$
strongly for every $1\leq i,j\leq n$, then, as $\hat{J}\Lambda_{\hat{\varphi}}(\hat{v}_{m})$
is a unit vector for every $m$, we obtain $1\leq n(2+\e)\e$. Taking
$\e$ small enough yields a contradiction. Therefore, $\Linfty{\hat{\G}}'\subseteq N\rtimes_{\a}\G$
cannot be $E_{\Linfty{\hat{\G}}'}$-weakly mixing.

\prettyref{enu:mixing_in_crossed_products} Assume that $\a$ is mixing,
that is, $(\theta\tensor\i)\left[\a(a)(b\tensor\one)\right]\in c_{0}(\G)$
for every $a,b\in\ker\theta$ (see \prettyref{eq:state_preserve_act_unitary_imp}).
To prove that $\Linfty{\hat{\G}}'\subseteq N\rtimes_{\a}\G$ is $E_{\Linfty{\hat{\G}}'}$-mixing,
fix a bounded sequence $\left(\hat{v}_{n}\right)$ in $\Linfty{\hat{\G}}$
converging weakly to zero, not necessarily of unitaries. Since $\linfty{\G}$
is an $\ell^{\infty}$-direct sum of (finite-dimensional) matrix algebras,
$c_{0}(\G)$ is the corresponding $c_{0}$-direct sum and $\linfty{\G}$
is in standard form on $\ltwo{\G}$, every $z\in c_{0}(\G)$ is compact,
so that $z\hat{J}\Lambda_{\hat{\varphi}}(\hat{v}_{n})\xrightarrow[n\to\infty]{}0$.
In particular, if $a,b\in\ker\theta$, then from \prettyref{eq:disc_crossed_prod_cond_exp}
we get 
\[
E_{\Linfty{\hat{\G}}'}\bigl[\a(a)(\one\tensor\hat{J}\hat{v}_{n}\hat{J})\a(b)\bigr]\Lambda_{\hat{\varphi}}(\one)=(\theta\tensor\i)\left[\a(a)(b\tensor\one)\right]\hat{J}\Lambda_{\hat{\varphi}}(\hat{v}_{n})\xrightarrow[n\to\infty]{}0.
\]
Hence $E_{\Linfty{\hat{\G}}'}\bigl[\a(a)(\one\tensor\hat{J}\hat{v}_{n}\hat{J})\a(b)\bigr]\to0$
strongly. Thus \prettyref{eq:mixing_inclusion} holds for every $x\in C$
and $y\in C^{*}$. As $C$ is dense in $N\rtimes_{\a}\G$ in the $*$-strong
topology, \prettyref{eq:mixing_inclusion} holds for every $x\in C$
and $y\in N\rtimes_{\a}\G$.\end{proof}
\begin{rem}
What we established in the proof of \prettyref{enu:mixing_in_crossed_products}
is formally stronger than mixing since the operators $\left(\hat{v}_{n}\right)$
are not assumed to be unitary. However, this is not surprising in
light of \citep[Theorem 3.3]{Cameron_Fang__Mukherjee__mix_subalg}.
It is interesting to check this result for general von Neumann algebras.
\end{rem}

\section{\label{sec:open_questions}Open questions}

The results of \prettyref{sec:equivalent_conditions} may open the
door to solving several questions.

In his celebrated proof of Szemer\'edi's theorem, Furstenberg \citep{Furstenberg__erg_behav}
established a multiple recurrence result that significantly extended
the Poincar\'{e} recurrence theorem, and led to a breadth of works
on related convergence questions. Weak mixing was a key idea in Furstenberg's
paper. Beyers, Duvenhage and Str{\"o}h \citep{Beyers_Duvenhage_Stroh__Szemeredi}
and Austin, Eisner and Tao \citep{Austin_Eisner_Tao_noncomm_erg_av}
considered the noncommutative case, namely in which the object that
is acted on is a finite von Neumann algebra, obtaining very interesting
partial results.
\begin{question}
\label{ques:Furstenberg}Is it possible to generalize results of \citep{Beyers_Duvenhage_Stroh__Szemeredi,Austin_Eisner_Tao_noncomm_erg_av},
and in particular \citep[Theorem 1.17]{Austin_Eisner_Tao_noncomm_erg_av},
to actions of LCQGs?
\end{question}
The main ingredients in the proof of \citep[Theorem 1.17]{Austin_Eisner_Tao_noncomm_erg_av}
include the von Neumann algebraic Jacobs--de Leeuw--Glicksberg theorem
of \citep{Niculescu_Stroh_Zsido__noncmt_recur} and the classical
van der Corput estimate. The starting point for answering \prettyref{ques:Furstenberg}
can be \prettyref{thm:Jacobs_de-Leeuw_Glicksberg} and a possible
generalization of the van der Corput estimate (for locally compact
groups, this was done in \citep{Beyers_Duvenhage_Stroh__Szemeredi}).
Note that even the case of groups other than $\Z$, or actions on
infinite von Neumann algebras, is still unknown. 

Two close questions concern formal strengthenings of the definition
of property (T) for LCQGs, whose statements or proofs are related
to weak mixing. In both, some separability assumption needs to be
made. Property (T) for discrete quantum groups was introduced by Fima
\citep{Fima__prop_T} and studied further by Kyed \citep{Kyed__cohom_prop_T_QG}.
For general LCQGs, see \citep[Section 6]{Daws_Fima_Skalski_White_Haagerup_LCQG}
and \citep{Chen_Ng__prop_T_LCQGs}.
\begin{question}
Does the Connes--Weiss theorem \citep{Connes_Weiss}, \citep[Theorem 6.3.4]{Bekka_de_la_Harpe_Valette__book},
characterizing property (T) in terms of strong ergodicity of measure-preserving
ergodic actions, generalize to LCQGs?
\end{question}
Generalizing the original proof, for example, would require a modification
of \citep[Proposition 3.2, (T2)]{Chen_Ng__prop_T_LCQGs} combined
with the construction of Vaes \citep[Proposition 3.1]{Vaes_strict_out_act}
that produces, from a unitary co-representation ``arising from an
orthogonal co-representation'', an action on a free Araki--Woods
factor preserving the free quasi-free state.
\begin{question}
Does the Bekka--Valette characterization \citep[Theorem 1]{Bekka_Valette__prop_T_amen_rep},
\citep[Theorem 2.12.9]{Bekka_de_la_Harpe_Valette__book} generalize
to LCQGs? That is, are the following conditions equivalent for a LCQG
$\G$?
\begin{enumerate}
\item $\G$ has property (T);
\item every unitary co-representation of $\G$ that has almost-invariant
vectors is not (strictly) weakly mixing.
\end{enumerate}
\end{question}
In \citep{Okayasu_Ozawa_Tomatsu_Haagerup_via_bimod}, Okayasu, Ozawa
and Tomatsu introduce strict mixing of bimodules (correspondences)
over von Neumann algebras. To every unitary co-representation of a
LCQG $\G$ they associate an $\Linfty{\hat{\G}}$-$\Linfty{\hat{\G}}$
bimodule, and vice versa, and establish that mixing of an element
of either of these classes implies mixing of the associated element
of the other \citep[Proposition 14]{Okayasu_Ozawa_Tomatsu_Haagerup_via_bimod}.
\begin{question}
Can one introduce a notion of weak mixing of bimodules over von Neumann
algebras (see, e.g., Peterson and Sinclair \citep[Definition 2.3]{Peterson_Sinclair__cocyc_superrig_Gauss_act})
that would be consistent as above with that of weak mixing for unitary
co-representations of LCQGs? 
\end{question}
\appendix

\section{Auxiliary results}

The following simple results, which are probably known, are used in
the paper. For a lucid account of unbounded operators on Hilbert spaces,
see the classic of Dunford and Schwartz \citep[Chapter XII]{DS2}.
Recall that over a given Hilbert space, a bounded operator $b$ is
said to \emph{commute} with an unbounded normal operator $n$ when
$bn\subseteq nb$; equivalently, when $b$ commutes with all spectral
projections of $n$ (see Fuglede \citep{Fuglede__comm_thm}).
\begin{lem}
\label{lem:action_with_invariant_state_J}Let $\a$ be an action of
a LCQG $\G$ on a von Neumann algebra $N$ that preserves a faithful
normal state $\theta$ of $N$. Then the unitary implementation $U\in B(\Ltwo{N,\theta})\tensorn\Linfty{\G}$
of $\a$ satisfies 
\[
(J_{\theta}\tensor\hat{J})U(J_{\theta}\tensor\hat{J})=U^{*}.
\]
\end{lem}
\begin{proof}
This is another by-product of the proof of \citep[Theorem A.1]{Runde_Viselter_LCQGs_Ergodic_Thy}
(see also the lines that precede it). Indeed, for $a\in D(\sigma_{-i}^{\theta})$
and $b\in D(\sigma_{i}^{\theta})$, we have 
\[
(\om_{\Lambda_{\theta}(a),\Lambda_{\theta}(b)}\tensor\i)(U)\hat{\nabla}^{1/2}\hat{J}\subseteq\hat{\nabla}^{1/2}\hat{J}(\om_{\Lambda_{\theta}(a^{*}),\Lambda_{\theta}(\sigma_{i}^{\theta}(b)^{*})}\tensor\i)(U^{*})
\]
(see \citep[equation succeeding (A.4) and (A.1), (A.2)]{Runde_Viselter_LCQGs_Ergodic_Thy};
we use $\hat{J},\hat{\nabla}$ in place of $I,L$). This was used
to prove that $U$ commutes with $\nabla_{\theta}\tensor\hat{\nabla}^{-1}$,
which implies that 
\[
\hat{\nabla}^{1/2}(\om_{\Lambda_{\theta}(\sigma_{-i/2}^{\theta}(a)),\Lambda_{\theta}(\sigma_{i/2}^{\theta}(b))}\tensor\i)(U)\hat{J}=\hat{\nabla}^{1/2}\hat{J}(\om_{\Lambda_{\theta}(a^{*}),\Lambda_{\theta}(\sigma_{i}^{\theta}(b)^{*})}\tensor\i)(U^{*})
\]
on $D(\hat{\nabla}^{1/2}\hat{J})=D(\hat{\nabla}^{-1/2})$. Since $\hat{\nabla}^{1/2}$
is strictly positive, we get 
\[
(\om_{\Lambda_{\theta}(\sigma_{-i/2}^{\theta}(a)),\Lambda_{\theta}(\sigma_{i/2}^{\theta}(b))}\tensor\i)(U)\hat{J}=\hat{J}(\om_{\Lambda_{\theta}(a^{*}),\Lambda_{\theta}(\sigma_{i}^{\theta}(b)^{*})}\tensor\i)(U^{*}).
\]
Recall that for every $c\in D(\sigma_{i/2}^{\theta})$, $J_{\theta}\Lambda_{\theta}(c)=J_{\theta}\nabla_{\theta}^{1/2}\nabla_{\theta}^{-1/2}\Lambda_{\theta}(c)=J_{\theta}\nabla_{\theta}^{1/2}\Lambda_{\theta}(\sigma_{i/2}^{\theta}(c))=\Lambda_{\theta}(\sigma_{-i/2}^{\theta}(c^{*}))$.
As a result,
\[
\hat{J}(\om_{J_{\theta}\Lambda_{\theta}(a^{*}),J_{\theta}\Lambda_{\theta}(\sigma_{i}^{\theta}(b)^{*})}\tensor\i)(U)\hat{J}=(\om_{\Lambda_{\theta}(a^{*}),\Lambda_{\theta}(\sigma_{i}^{\theta}(b)^{*})}\tensor\i)(U^{*}).
\]
The density of $\left\{ \Lambda_{\theta}(a^{*}):a\in D(\sigma_{-i}^{\theta})\right\} $
and $\left\{ \Lambda_{\theta}(\sigma_{i}^{\theta}(b)^{*}):b\in D(\sigma_{i}^{\theta})\right\} $
in $\Ltwo{N,\theta}$ gives the result.\end{proof}
\begin{lem}
\label{lem:abc}Let $\H_{1},\H_{2}$ be Hilbert spaces and $a,c\in B(\H_{1},\H_{2})$.
Then $a^{*}c,ac^{*}$ are selfadjoint, $\ker a\subseteq\ker c$ and
$\ker a^{*}=\ker c^{*}$ if and only if there exists a (generally
unbounded) injective selfadjoint operator $b$ over $\H_{2}$ commuting
with $aa^{*}$ such that $c=ba$.\end{lem}
\begin{proof}
Define $b(a\z):=c\z$, $\z\in\H_{1}$. By the selfadjointness of $a^{*}c$,
the (generally unbounded) operator $b$ is symmetric, thus closable,
over the Hilbert space $\K:=\overline{\Img a}=\overline{\Img c}\subseteq\H_{2}$.
As $a^{*}c,ac^{*}$ are selfadjoint, one calculates that $aa^{*}b\subseteq baa^{*}$,
and so $aa^{*}\overline{b}\subseteq\overline{b}aa^{*}$. If $\eta\in D(b^{*})$,
then for all $\z\in\H_{1}$, $\left\langle c\z,\eta\right\rangle =\left\langle ba\z,\eta\right\rangle =\left\langle \z,a^{*}b^{*}\eta\right\rangle $,
so that $c^{*}\eta=a^{*}b^{*}\eta$. If now $b^{*}\eta=\pm i\eta$,
then $\left\Vert c^{*}\eta\right\Vert ^{2}=\mp i\left\langle c^{*}\eta,a^{*}\eta\right\rangle =\mp i\left\langle ac^{*}\eta,\eta\right\rangle \in i\R$
as $ac^{*}$ is selfadjoint. Hence $\eta\in\ker c^{*}$, so that $\eta=0$
as $\eta\in\K$. Therefore, $b$ is essentially selfadjoint by von
Neumann's decomposition of $D(b^{*})$ \citep[Lemma XII.4.10]{DS2}.
Since the range of $b$ is dense in $\K$, its closure $\overline{b}$
is injective. The desired $b$ can now be defined as $\overline{b}$
on $\K$ and $\one$ on $\K^{\perp}$.

Conversely, if such $b$ exists, then clearly $\ker a=\ker c$ and
$a^{*}c$ is selfadjoint. We have $\overline{\Img c}=\overline{\Img ba}\supseteq\overline{\Img baa^{*}}=\overline{\Img aa^{*}b}=\overline{\Img aa^{*}}=\overline{\Img a}$
since $b$ commutes with $aa^{*}$ and is injective. Also, for the
same reasons, for each $\z\in\H_{1}$ there exists a sequence $\left(\eta_{n}\right)$
in $D(b)$ such that $ca^{*}\eta_{n}=baa^{*}\eta_{n}=aa^{*}b\eta_{n}\to a\z$,
implying that $\overline{\Img a}\subseteq\overline{\Img c}$. Thus
$\ker c^{*}=\ker a^{*}$. Furthermore, $ca^{*}=baa^{*}$ is selfadjoint
as $aa^{*},b$ are commuting selfadjoint operators.\end{proof}
\begin{lem}
Let $T$ be a closed, densely-defined operator on a Hilbert space
$\H$. If $\left(\eta_{\iota}\right)$ is a net in $D(T)$ and $\eta\in\H$
are such that $\eta_{\iota}\to\eta$ weakly in $\H$ and $\left(T\eta_{\iota}\right)$
is bounded by $C$, then $\eta\in D(T)$ and $\left\Vert T\eta\right\Vert \leq C$.\end{lem}
\begin{proof}
For every $\z\in D(T^{*})$, $\left\langle T^{*}\z,\eta\right\rangle =\lim_{\iota}\left\langle \z,T\eta_{\iota}\right\rangle $
has absolute value at most $C\left\Vert \z\right\Vert $. Since $T=T^{**}$,
we infer that $\eta\in D(T)$ and $\left\Vert T\eta\right\Vert \leq C$.\end{proof}
\begin{lem}
\label{lem:action_conv_bdd}Let $\tau$ be an action of $\R$ on a
von Neumann algebra $M$ and let $z\in\C$. If a net $\left(a_{\iota}\right)$
in $D(\tau_{z})$ converges weakly to $a\in M$ and $\left(\tau_{z}(a_{\iota})\right)$
is bounded by $C$, then $a\in D(\tau_{z})$ and $\left\Vert \tau_{z}(a)\right\Vert \leq C$.\end{lem}
\begin{proof}
Representing $M$ standardly on a Hilbert space $\H$, there exists
a (generally unbounded) strictly positive operator $T$ on $\H$ such
that $\tau_{t}=\Ad{T^{it}}$ for all $t\in\R$ (by \citep[Corollary 3.6]{Haagerup__standard_form}
and Stone's theorem). Let $\z\in D(T^{-iz})$. For every $\iota$,
let $\eta_{\iota}:=a_{\iota}T^{-iz}\z=T^{-iz}\tau_{z}(a_{\iota})\z$.
Then $\eta_{\iota}\to aT^{-iz}\z$ weakly and the net $(T^{iz}\eta_{\iota})=(\tau_{z}(a_{\iota})\z)$
is bounded by $C\left\Vert \z\right\Vert $. By the previous lemma,
$aT^{-iz}\z\in D(T^{iz})$ and $\left\Vert T^{iz}aT^{-iz}\z\right\Vert \leq C\left\Vert \z\right\Vert $.
Consequently, $b:=\overline{T^{iz}aT^{-iz}}$ satisfies $\left\Vert b\right\Vert \leq C$.
It is standard that $b\in M$. For every $\z\in D(T^{-iz})$, we have
$aT^{-iz}\z=T^{-iz}T^{iz}aT^{-iz}\z=T^{-iz}b\z$, whence $a\in D(\tau_{z})$
and $\tau_{z}(a)=b$.
\end{proof}

\section*{Acknowledgments}

We thank Xiao Chen, Adam Skalski, Piotr M.~So{\ldash}tan and Nicolaas
Spronk for their comments and suggestions. In particular, we are grateful
to PMS for drawing our attention to relevant parts of his paper \citep{Soltan__quantum_Bohr_comp},
and to NS for suggesting the connection with \citep{Das_Daws__quantum_Eberlein}.
We are indebted to the referee of a previous version of this paper
for spotting a critical mistake, and to the referee of this version
for his/her suggestions.

\bibliographystyle{amsplain}
\bibliography{LCQG_weak_mixing}

\end{document}